\documentclass[a4paper, 11pt]{article}

\usepackage{fancyhdr}
\usepackage{extramarks}
\usepackage{amsmath}
\usepackage{amssymb}
\usepackage{amsthm}
\usepackage{amsfonts,url}
\usepackage{tikz}
\usepackage[noend]{algpseudocode}
\usepackage{algorithmicx,algorithm}
\usepackage{algpseudocode}
\usepackage{multirow}
\usepackage{subfigure,color}
\usepackage{epstopdf}
\usepackage{graphicx}
\usepackage{animate}
\graphicspath{{figures/}}

\numberwithin{equation}{section}
\usetikzlibrary{automata,positioning}

\newtheorem{theorem}{\textbf{Theorem}}

\newtheorem{definition}{\textbf{Definition}}
\newtheorem{lemma}{\textbf{Lemma}}

\newtheorem{remark}{\textbf{Remark}}
\newtheorem{example}{\textbf{Example}}

\title{Gauss Newton method for solving variational problems of PDEs with neural network discretizaitons}
\author{Wenrui  Hao\thanks{Department of Mathematics, Pennsylvania State University, USA   (email{wxh64@psu.edu}).}
\and Qingguo Hong\thanks{Department of Mathematics and Statistics, Missouri University of Science and Technology, USA.}
\and Xianlin Jin\thanks{School of Mathematical Sciences, Peking University, Beijing, China.}
}

\begin{document}

\maketitle
%\tableofcontents 
\begin{abstract}
The numerical solution of differential equations using machine learning-based approaches has gained significant popularity. Neural network-based discretization has emerged as a powerful tool for solving differential equations by parameterizing a set of functions. Various approaches, such as the deep Ritz method and physics-informed neural networks, have been developed for numerical solutions. Training algorithms, including gradient descent and greedy algorithms, have been proposed to solve the resulting optimization problems.

In this paper, we focus on the variational formulation of the problem and propose a Gauss-Newton method for computing the numerical solution. We provide a comprehensive analysis of the superlinear convergence properties of this method, along with a discussion on semi-regular zeros of the vanishing gradient. Numerical examples are presented to demonstrate the efficiency of the proposed Gauss-Newton method.

\end{abstract}

% REQUIRED


\section{Introduction}
The use of machine learning-based approaches in the computational mathematics community has witnessed significant growth in recent years, particularly in the numerical solution of differential equations. Neural network-based discretization has emerged as a revolutionary tool for solving differential equations \cite{dissanayake1994neural, han2018solving,sirignano2018dgm} and for discovering the underlying physics from experimental data \cite{raissi2019physics}. This approach has been successfully applied to a wide range of practical problems with remarkable success \cite{cai2022physics,gu2021selectnet,huang2022hompinns,pang2019fpinns}. One of the key advantages of this approach is that neural networks can alleviate, and in some cases overcome, the curse of dimensionality associated with high-dimensional problems \cite{khoo2021solving, han2018solving,lu2022priori}. This can be attributed to the dimension-independent approximation properties of neural networks \cite{klusowski2018approximation}, which have been compared to traditional methods such as finite element methods (FEMs) and other approximation techniques in the field of approximation theory \cite{ lu2021deep, shen2022optimal,siegel2022high,CiCP-28-1707}.

{Discretizing PDEs through neural networks involves parameterizing a set of functions aimed at solving these PDEs. To illustrate this concept, we consider the following Laplace's equation:

\[
\begin{cases}
-\Delta v = f(v) & \text{in } \Omega \\
\frac{\partial v}{\partial n} = 0 & \text{on } \partial \Omega
\end{cases}
\]

Here, $\Omega \subset \mathbb{R}^d$ and $\partial \Omega$ is the boundary of the domain. The associated energy functional is expressed as:

\[
\min_{w} J(w) = \int_{\Omega} |\nabla w|^2 - G(w), \, dx
\]
where $G'(w)=f(w)$.
Three primary approaches exist for solving numerical solutions based on the neural network discretization:

\begin{itemize}
\item \textbf{Variational Energy Minimization (Deep Ritz Method):}
   The first approach, known as the deep Ritz method \cite{yu2018deep}, involves minimizing the variational energy:

   \[
   \min_{\boldsymbol{\theta}} L(\boldsymbol{\theta}) = \int_{\Omega} |\nabla u(x ; \boldsymbol{\theta})|^2 - G(u(x ; \boldsymbol{\theta})) \, dx
   \]

\item \textbf{$L_2$ Residual Minimization:}
   The second approach, widely applied in Physics-Informed Neural Networks (PINNs) \cite{raissi2019physics}, minimizes the $L_2$ residual of PDEs and boundary conditions:

   \[
   \min_{\boldsymbol{\theta}} \sum_i \left\| \Delta u(x_i ; \boldsymbol{\theta}) + f(u(x_i ; \boldsymbol{\theta})) \right\|^2 + \sum_j \left\| \frac{\partial u(x_j ; \boldsymbol{\theta})}{\partial n} \right\|^2
   \]

\item \textbf{System of Nonlinear Equations:}
   The third approach involves solving a discretized system of nonlinear equations \cite{chen2022randomized}:

   \[
   \boldsymbol{F}(\boldsymbol{\theta}) = \begin{cases}
   \Delta u(x_i ; \boldsymbol{\theta}) + f(u(x_i ; \boldsymbol{\theta})), & i = 1, \cdots, N \\
   \frac{\partial u(x_j ; \boldsymbol{\theta})}{\partial \mathrm{n}}, & j = 1, \cdots, n
   \end{cases} = \mathbf{0}
   \]

In this context, the second approach is related to the third approach and can be expressed as $\min_{\boldsymbol{\theta}} \|\boldsymbol{F}(\boldsymbol{\theta})\|_2^2$.
\end{itemize}
}

Researchers have made efforts to bound the errors associated with these approaches. For instance, studies have shown that gradient descent applied to a wide enough network can reach a global minimum \cite{allen2019convergence, arora2019fine,zou2020gradient}. The convergence of stochastic gradient descent (SGD) and Adam optimizer \cite{kingma2014adam} has also been analyzed in Fourier space, revealing that the error converges more rapidly in the lowest frequency modes. This observation is known as the frequency principle or spectral bias of neural network training \cite{rahaman2019spectral}. Furthermore, a novel greedy training algorithm has been devised for shallow neural networks in order to numerically determine their theoretical convergence rate \cite{siegel2023greedy}.

In practice, Adam or SGD are commonly used to solve the resulting optimization problems in both variational and L2-minimization approaches. Additionally, randomized Newton's method has been developed to achieve faster local convergence by solving the system of nonlinear equations \cite{chen2022randomized}. Gauss-Newton method has also been employed to speed up the computation of the L2-minimization approach \cite{cai2019gram,muller2023achieving}. However, these training algorithms cannot be directly applied to the variational form of the problem.
In this paper, we will develop a Gauss-Newton method for the variational form and provide a convergence analysis for this proposed method.

The remaining sections of the paper are organized as follows:
In Section \ref{basic-section}, we introduce the problem setup and discuss the class of elliptic PDEs we will be solving.
Section \ref{GNM} provides an overview of Gauss-Newton methods for solving PDEs in both variational and L2-minimization forms.
The property of local minimizes, namely, semi-regular zeros of the vanishing gradient, is discussed in Section \ref{sec:semiregular}.
In Section \ref{sec:con}, we present a comprehensive analysis of the convergence properties of the Gauss-Newton method for the variational form, as well as its variant, the random Gauss-Newton method. 
Several numerical examples are presented in Section \ref{sec:num} to illustrate the efficiency of the proposed Gauss-Newton method.
Finally, in Section \ref{con}, we conclude the paper, by summarizing the key findings and discussing potential avenues for future research.

\section{Problem setup}\label{basic-section}
We consider the following second-order elliptic equation 
\begin{align}\label{EP1}
\mathcal{L}v = f(v) \text{ in } \Omega,
\end{align}
where $\Omega \in \mathbb{R}^{d}$ is an open subset, $d \geq 1$. Here $\mathcal{L}$ is the second-order elliptic operator defined by 
\begin{align}\label{OperatorL}
\mathcal{L} = -\sum_{i=1}^{d}\sum_{j=1}^{d} a_{i,j} \dfrac{\partial }{\partial x_i}\dfrac{\partial }{\partial x_j} + \sum_{k=1}^{d} b_{k} \dfrac{\partial }{\partial x_k} + c .
\end{align} 
To ensure the existence and uniqueness of the weak solution for problem (\ref{EP1}), certain assumptions need to be made on the operator $\mathcal{L}$. Specifically, the following conditions should be satisfied \cite{evans2022partial}: 
\begin{align}\label{Assumption1}
& a_{i,j} \in L^{\infty}(\Omega), \quad 1\leq i,j\leq d,
\end{align}

which means that the coefficients of the operator are bounded in $L^\infty(\Omega)$.

Moreover, there should exist positive constants $\lambda$ and $\Lambda$ such that
\begin{align}\label{Assumption2}
& \lambda \vert \boldsymbol{\xi} \vert^2 \leq \sum_{i=1}^{d}\sum_{j=1}^{d} a_{i,j} \xi_i\xi_j \leq \Lambda \vert \boldsymbol{\xi} \vert^2, \quad \forall \boldsymbol{\xi}\in \mathbb{R}^{d},\quad \forall x\in \Omega.
\end{align}

This means that the operator is uniformly elliptic, i.e., it satisfies a strong ellipticity condition, with ellipticity constants $\lambda$ and $\Lambda$. These assumptions ensure that the weak solution to the problem (\ref{EP1}) exists and is unique.
To simplify notation, we consider the following second-order partial differential equation: 
\begin{align} \label{2ndPDE}
-a\Delta v(x) +cv(x) &= f(x), \text{ in } \Omega\\
\dfrac{\partial v}{\partial n } &=0, \text{ on } \partial \Omega
\end{align}
{ where $a> 0$ and $c\geq 0$ are} constants such that the conditions (\ref{Assumption1})-(\ref{Assumption2}) are satisfied. It should be noted that all the results presented in this paper can be extended to the more general form of the operator $L$ defined in (\ref{OperatorL}), with Dirichlet boundary conditions.

We define the admissible set $V$ on $\Omega$ as
$V = H^1(\Omega)$
 and transform the PDE (\ref{2ndPDE}) into an energy minimization problem, given by:
{\begin{align}\label{minEnergy}
\min_{w\in V} \mathcal{J}(w) = \int_{\Omega}\left( \dfrac{a}{2} \vert \nabla w(x) \vert^2 + \dfrac{c}{2} w(x)^2 - f(x)w(x \right) dx.
\end{align}
Here, $\mathcal{J}(w)$ is the energy functional, and the minimizer $v = \displaystyle\arg\min_{w\in V} \mathcal{J}(w) $ of (\ref{minEnergy}) satisfies the PDE (\ref{2ndPDE}). { It is important to note that minimizing (\ref{minEnergy}) is equivalent to solving the PDE (\ref{2ndPDE}).}

The minimization problem can be solved using various learning algorithms, such as the Deep Ritz method \cite{yu2018deep}, which utilizes deep neural networks (DNNs). In practice, the energy integral in (\ref{minEnergy}) can be computed using numerical quadrature methods, such as the Gauss quadrature and the Monte Carlo method \cite{hammersley2013monte}.

\section{Gauss-Newton method for the variational problem}\label{GNM}
{ In this section, without loss of generality, we will introduce the Gauss-Newton method for solving the energy minimization problem (\ref{minEnergy}) with $a=1$ and $c=1$}. Let $u(x,\theta)$ be a learning model to approximate the exact solution $w(x)$, for example, $u(x,\theta)$ can be the finite element function or the neural network function. 
{
More specifically, we consider a neural network discretization, consisting of a sequence of fully connected layers from the input \(x \in \mathbb{R}^d\) to the output \(u \in \mathbb{R}\). Mathematically, a neural network with \(J\) hidden layers can be expressed as follows:
\begin{equation}
    u(x; \theta) = W_Jh_{J-1} + b_J, \quad h_i = \sigma(W_i h_{i-1} + b_i), \quad i \in \{1, \ldots, J-1\}, \quad \text{and } h_0 = x,
\end{equation}
where \(W_i \in \mathbb{R}^{d_{i}\times d_{i-1}}\) is the weight, \(b_i \in \mathbb{R}^{d_{i}}\) is the bias, \(d_i\) is the width of the \(i\)-th hidden layer, and \(\sigma\) is the activation function (e.g., ReLU or the sigmoid activation functions). For simplicity, we denote $m=d_1+d_2+\cdots+d_{J-1}$ and all the parameters of the neural network, including weights and biases, as \(\theta\). Then we define deep finite neuron functions with the activation function \(\sigma = \text{ReLU}^k\) as \(DNN_J\) with \(J\) layers.
}
In this paper, we will focus on a smaller admissible set, i.e., the DNN function space, where the loss function takes the form:
\begin{align}\label{lossfuncL}
L(\theta) = \int_{\Omega} \dfrac{1}{2} \vert\nabla u(x,\theta)\vert^2 + \dfrac{1}{2} u(x,\theta)^2 - f(x)u(x,\theta) dx.
\end{align} The gradient of $L(\theta)$ is 
\begin{align}\label{gradL}
\nabla_{\theta}L(\theta) = \int_{\Omega} \nabla_{\theta} \nabla u(x,\theta) \nabla u(x,\theta) + u(x,\theta) \nabla_{\theta}u(x,\theta) - f(x)\nabla_{\theta}u(x,\theta) dx
\end{align}
The Hessian of $L(\theta)$ is as follows:
\begin{align}\label{HessianL}
\boldsymbol{\operatorname{H}}(\theta) &= \underbrace{\int_{\Omega} \nabla_{\theta} \nabla u(x,\theta) \cdot
\nabla_{\theta}  \nabla u(x,\theta)^{T}
+ \nabla_{\theta}u(x,\theta) \cdot \nabla_{\theta}u(x,\theta)^{T} dx}_{\boldsymbol{J}(\theta)} \\
&+ \underbrace{\int_{\Omega} \nabla^2_{\theta} \nabla u(x,\theta) \cdot \nabla u(x,\theta) + u(x,\theta) \nabla^2_{\theta}u(x,\theta) - f(x)\nabla^2_{\theta}u(x,\theta) dx
}_{\boldsymbol{Q}(\theta)}, 
\end{align}
{where $\boldsymbol{J}(\theta)$ and $\boldsymbol{Q}(\theta)$ denote terms involving the first- and second-order derivative information of $\theta$, respectively. }

%Given $d, J\in\mathbb{N}^+$, 
%$
%n_1,\dots,n_{J}\in\mathbb{N} \mbox{ with }n_0=d, n_{J+1}=1, 
%$
%\begin{equation}\label{thetamap}
%\theta^i(x)=\omega_i\cdot x + b_i,\quad \omega_i\in \mathbb{R}^{n_{i+1}\times n_i},\ b\in \mathbb{R}^{n_{i+1}},
%\end{equation}
%and the activation function ReLU$^k,k\ge 1$, define a deep finite neuron function $u(x,\theta)$ from $\mathbb{R}^d$ to $\mathbb{R}$  as follows:
%\begin{align*}
%u^0(x)   &=\theta^0(x) \\ 
%u^{i}(x) &= [  \theta^{i} \circ \sigma ](u^{i-1}(x)) \quad i = 1:J \\
%u(x,\theta) &= f^J(x).
%\end{align*}
%The following more concise notation is often used in computer science literature:
%\begin{equation}
%\label{compress-dnn}
%u(x,\theta) = \theta^{J}\circ \sigma \circ \theta^{J-1} \circ \sigma \cdots \circ \theta^1 \circ \sigma \circ \theta^0(x),
%\end{equation}
%here $\theta^i: \mathbb{R}^{n_{i}}\to\mathbb{R}^{n_{i+1}}$ are linear
%functions as defined in \eqref{thetamap}.  Such a deep neutral network
%has $(J+1)$-layer DNN, namely $J$-hidden layers. The size of
%this deep neutral network is $m:=n_1+\cdots+n_{J}$.

%Based on these notation and connections, define deep finite neuron
%functions with activation function $\sigma=\hbox{ReLU}^k$ by
%\begin{equation}
%\label{NNL}
%DNN_J=\bigg\{ u^{J}(x) = \theta^J (x^{J}), 
% \mbox{ with } W^i\in \mathbb R^{n_{i+1}\times
%	n_{i}}, b^i\in\mathbb R^{n_i}, i=0:J, n_0=d, n_{J+1}=1\bigg\}.
%\end{equation}
%}

Then we estimate the second derivative matrix by the following lemma \cite{barron1993universal,siegel2020approximation,siegel2022sharp,CiCP-28-1707}.

\begin{lemma}\label{lemma4}
For $\forall \epsilon > 0$, there exist $\delta > 0$, $J\in \mathbb{N}^{+}$ and $m\in \mathbb{N}^{+}$, 
such that $ \theta\in \mathbb{R}^m$, $DNN_J\subset H^1(\Omega)$, and for $
\Vert \theta-\theta^* \Vert < \delta$, $ \theta^*=\arg\min L(\theta)$, if $ D^2_{\theta} u(x,\theta) \in H^1(\Omega)$, then it holds 
$\Vert \boldsymbol{Q}(\theta)\Vert < \epsilon.$
\end{lemma}
\begin{proof}
First, we can rewrite $\boldsymbol{Q}(\theta)$ using integration by parts as   
\begin{align}
\boldsymbol{Q}(\theta)&=\int_{\Omega} \nabla^2_{\theta} \nabla u(x,\theta) \cdot \nabla u(x,\theta) + u(x,\theta) \nabla^2_{\theta}u(x,\theta) - f(x)\nabla^2_{\theta}u(x,\theta) dx\\
&=\int_{\Omega} \nabla^2_{\theta} u(x,\theta)(-\Delta u(x,\theta) + u(x,\theta) -f(x)) dx+\int_{\partial \Omega} \nabla^2_{\theta} u(x,\theta) \frac{\partial u(x,\theta)}{\partial n}ds.
\end{align}
Given that the true solution $v(x)$ satisfies:\begin{align} 
-\Delta v(x) +v(x) &= f(x), \text{ in } \Omega\\
\dfrac{\partial v}{\partial n } &=0, \text{ on } \partial \Omega.
\end{align}
{ We can rewrite $\boldsymbol{Q}(\theta)$ as follows:}
\begin{align}
\boldsymbol{Q}(\theta)
=&\int_{\Omega} \nabla^2_{\theta} u(x,\theta)(-\Delta u(x,\theta) + u(x,\theta) -f(x)) dx+\int_{\partial \Omega} \nabla^2_{\theta} u(x,\theta) \frac{\partial u(x,\theta)}{\partial n}ds\\
=&\int_{\Omega} \nabla^2_{\theta} u(x,\theta)(-\Delta u(x,\theta) + u(x,\theta) + \Delta v(x) -v(x))dx\\
&+\int_{\partial \Omega} \nabla^2_{\theta} u(x,\theta)\left(\frac{\partial u(x,\theta)}{\partial n}-\frac{\partial v(x)}{\partial n}\right)ds,
\end{align}
which implies \begin{align}
\|\boldsymbol{Q}(\theta)\|
\le &\|\nabla^2_{\theta} u(x,\theta)\|_{H^{1}(\Omega)} \big(\|\Delta u(x,\theta)-\Delta v(x)\|_{H^{-1}(\Omega)} +\|u(x,\theta)- v(x)\|_{H^{-1}(\Omega)}\big)\\
&+\|\nabla^2_{\theta} u(x,\theta)\|_{H^{\frac{1}{2}}(\partial\Omega)}
\bigg\|\frac{\partial u(x,\theta)}{\partial n}-\frac{\partial v(x)}{\partial n}\bigg\|_{H^{-\frac{1}{2}}(\partial\Omega)}.
\end{align}
Hence, by trace theorem, we have 
\begin{align}
\|\boldsymbol{Q}(\theta)\|
\le &C \|\nabla^2_{\theta} u(x,\theta)\|_{H^{1}(\Omega)} \|u(x,\theta)-v(x)\|_{H^{1}(\Omega)}.
\end{align}
{Given that $\theta^*=\arg\min L(\theta)$ and $u(x,\theta^*)$ is the approximation solution for $v(x)$, we have  the following error estimates \cite{CiCP-28-1707}:
\begin{align}
\|u(x,\theta^*)-v(x)\|_{H^{1}(\Omega)}\le 
\inf_{u(x,\theta)\in DNN_J} 
\|u(x,\theta)-v(x)\|_{H^{1}(\Omega)}. \nonumber
\end{align} 
This leads to 
\begin{align}
\|\boldsymbol{Q}(\theta^*)\|
\le &C \|\nabla^2_{\theta} u(x,\theta^*)\|_{H^{1}(\Omega)} \|u(x,\theta^*)-v(x)\|_{H^{1}(\Omega)}\nonumber\\
\le& C \|\nabla^2_{\theta} u(x,\theta^*)\|_{H^{1}(\Omega)} \inf_{u(x,\theta)\in DNN_J} 
\|u(x,\theta)-v(x)\|_{H^{1}(\Omega)}.\nonumber
\end{align}
 Furthermore, using the approximation results of neural networks from \cite{barron1993universal, siegel2020approximation, siegel2022sharp} and considering the assumption $D^2_{\theta} u(x,\theta^*) \in H^1(\Omega)$, we have for any $\epsilon > 0$, there exist $J\in \mathbb{N}^{+}$ and $m\in \mathbb{N}^{+}$, 
such that $\theta\in \mathbb{R}^m$, $DNN_J\subset H^1(\Omega)$ and
\begin{align}
\inf_{u(x,\theta)\in DNN_J} 
\|u(x,\theta)-v(x)\|_{H^{1}(\Omega)}< \frac{\epsilon}{2C \|\nabla^2_{\theta} u(x,\theta^*)\|_{H^{1}(\Omega)}},\nonumber
\end{align}
implying $\|\boldsymbol{Q}(\theta^*)\|< \frac{\epsilon}{2}$. 
Finally, leveraging the continuity of $|\boldsymbol{Q}(\theta)|$ concerning $\theta$, we establish that for every $\epsilon > 0$, there exists $\delta > 0$ such that for $\Vert \theta-\theta^* \Vert < \delta$, we have $|\boldsymbol{Q}(\theta)-\boldsymbol{Q}(\theta^*)|<\frac{\epsilon}{2}$. Consequently, the desired result is attained through the application of the triangle inequality.
}
\end{proof}
   % {\bf Qingguo, can we prove %this?}

Hence, it is reasonable to consider the first-order approximation of the Hessian, i.e., $\boldsymbol{J}(\theta) \approx \boldsymbol{\operatorname{H}}(\theta)$. The Gauss-Newton method for the variational problem is then given by
\begin{align}\label{varIter}
\theta_{k+1} = \theta_k - \boldsymbol{J} (\theta_k)^{\dagger}  \nabla_{\theta}L(\theta_k), \quad k=0,1,2,\cdots.
\end{align}

\subsection{Gauss-Newton method for solving the L2 minimization problem}
In this subsection, we will review the Gauss-Newton method for solving the L2 minimization problem, namely,
\begin{align}\label{L2min}
\min_{\theta} \frac{1}{2} \|\boldsymbol{F}(\theta)\|_2^2,
\end{align}
where 
{
\begin{align}\label{collocation}
\boldsymbol{F}(\theta) = 
\begin{pmatrix}
-\Delta u(x_1,\theta) + u(x_1,\theta) - f(x_1)& \\ \vdots & \\
-\Delta u(x_N,\theta) + u(x_N,\theta) - f(x_N) & \\
\nabla u(x_1^b,\theta) \cdot \boldsymbol{n} & \\
%\nabla u(x_2^b,\theta) \cdot \boldsymbol{n} & \\
\vdots & \\
\nabla u(x_n^b,\theta) \cdot \boldsymbol{n} 
\end{pmatrix},
\end{align}
}
with collocation points $\lbrace x_1,x_2,\cdots, x_N \rbrace \subset \Omega$ and $\lbrace x^b_1, x^b_2, \cdots, x^b_n \rbrace \subset \partial \Omega$. Here, $\boldsymbol{n}$ is the outer normal unit vector of $\partial \Omega$, and $\theta =\lbrace \theta_1, \theta_2,\cdots, \theta_m \rbrace$ represents the parameters in the learning model $u(x,\theta)$.
Therefore, an optimal choice of $\theta$ is computed by the Gauss-Newton method as
\begin{align}\label{colloIter}
\theta_{k+1} = \theta_k - (\boldsymbol{\operatorname{JF}} (\theta_k))^{\dagger}  \boldsymbol{F}(\theta_k),\quad k=0,1,2,\cdots
\end{align}
where $^{\dagger}$ denotes the Moore-Penrose inverse and $\boldsymbol{\operatorname{JF}}(\theta) $ is the Jacobi matrix defined as 
{
\begin{align}\label{Jacobian}
\boldsymbol{\operatorname{JF}}(\theta) =
\begin{bmatrix}
 -\nabla_{\theta} \Delta u(\boldsymbol{x},\theta)^{T} + \nabla_{\theta} u(\boldsymbol{x},\theta)^{T} \\
\left(\nabla_{\theta} \nabla u(\boldsymbol{x}^b,\theta)\cdot \boldsymbol{n}\right)^{T} \\
\end{bmatrix}
\in \mathbb{R}^{(N+n)\times m}
\hbox{~and~} \nabla_{\theta} \nabla=\begin{bmatrix}
\partial_{x_1} \partial_{\theta_1} & \cdots & \partial_{x_d} \partial_{\theta_1} & \\
%\partial_{x_1} \partial_{\theta_2} & \cdots & \partial_{x_d} \partial_{\theta_2} & \\
\vdots                                           &                                          \ddots  &\vdots &\\
\partial_{x_1} \partial_{\theta_m} & %\partial_{x_2} \partial_{\theta_m} 
\cdots & \partial_{x_d} \partial_{\theta_m} & 
\end{bmatrix}.
\end{align}
}

\subsection{The consistency between Gauss-Newton methods for L2 minimization and variational problems}

By applying the divergence theorem to $\boldsymbol{J}(\theta)$, we obtain:
\begin{align}\label{JofTheta}
 \boldsymbol{J}(\theta) &= \int_{\Omega} \nabla_{\theta} u(x,\theta) \cdot \left( -\nabla_{\theta} \Delta u(x,\theta) + \nabla_{\theta} u(x,\theta) \right)^{T}  dx  + \int_{\partial \Omega} \nabla_{\theta} u(x,\theta) \cdot \dfrac{\partial  \nabla_{\theta} u(x,\theta)}{\partial \boldsymbol{n}}^{T}dS.
\end{align}
If we compute all integrals using numerical methods, e.g., the Gaussian quadrature rule, we can derive the condition for the consistency of the Gauss-Newton method in (\ref{varIter}) with the Gauss-Newton method for the L2 minimization problem in (\ref{colloIter}). Denote the grid points in the domain $\Omega$ as $\boldsymbol{x} = (x_1, x_2, \cdots, x_N)^{T}$ and the corresponding  weights as $\boldsymbol{w} = (w_1,w_2,\dots,w_N)^{T}$. Then, we can write the first part of (\ref{JofTheta}) as:
\begin{align}
 &\ \int_{\Omega}  \nabla_{\theta} u(x,\theta) \cdot \left( -\nabla_{\theta} \Delta u(x,\theta) + \nabla_{\theta} u(x,\theta)  \right)^{T} dx = \sum_{i=1}^{N}{w_i}\nabla_{\theta} u(x_i,\theta) \cdot \left( -\nabla_{\theta} \Delta u(x_i,\theta) + \nabla_{\theta} u(x_i,\theta)   \right)^{T}.
\end{align}
Similarly, with grid points on the boundary $\boldsymbol{x^b} = (x^b_1, x^b_2, \cdots, x^b_n)^{T}$ and weights $\boldsymbol{w^b} = (w^b_1, w^b_2, \cdots, w^b_n)^{T}$, we can write the second part of (\ref{JofTheta}) as:
\begin{align}
 \int_{\partial \Omega}   \nabla_{\theta} u(x,\theta) \cdot \dfrac{\partial  \nabla_{\theta} u(x,\theta)}{\partial \boldsymbol{n}}^{T} dS  = \sum_{j=1}^{n} w^b_j \nabla_{\theta} u(x_j,\theta) \cdot \dfrac{\partial  \nabla_{\theta} u(x_j,\theta)}{\partial \boldsymbol{n}}^{T}. 
\end{align}

Thus, we have 
\begin{align}
 \boldsymbol{J}(\theta) =\boldsymbol{G} \cdot \boldsymbol{\operatorname{JF}(\theta)} \hbox{~where~} 
\boldsymbol{G}=  \begin{bmatrix}
  {w_1} \nabla_{\theta} u(x_1,\theta)  & \cdots & {w_N} \nabla_{\theta}u(x_N,\theta)  & 
{w^b_1} \nabla_{\theta}  u(x^b_1,\theta)  & \cdots &
{w^b_n} \nabla_{\theta}  u(x^b_n,\theta) 
  \end{bmatrix}
\end{align}
Similarly, we can rewrite the gradient defined in (\ref{gradL}) as
\begin{align}
\nabla_{\theta} L(\theta) &= \int_{\Omega}\nabla_{\theta} u(x,\theta) \left( -\Delta u(x,\theta) + u(x,\theta) -f(x) \right)   dx + \int_{\partial \Omega}\nabla_{\theta} u(x,\theta)  \dfrac{\partial \nabla_{\theta}u(x,\theta) }{\partial \boldsymbol{n}} dS = \boldsymbol{G}\cdot \boldsymbol{F}(\boldsymbol{x},\theta) 
\end{align}
{
If $\boldsymbol{G}$ has linearly independent columns and $\boldsymbol{\operatorname{JF}}(\theta)$ has linearly independent rows, then we have $(\boldsymbol{G} \cdot \boldsymbol{\operatorname{JF}}(\theta))^{\dagger}=(\boldsymbol{\operatorname{JF}}(\theta))^{\dagger}\cdot\boldsymbol{G}^{\dagger}$  and $\boldsymbol{G}^{\dagger} \boldsymbol{G} = \boldsymbol{I} \in \mathbb{R}^{(N+n)\times (N+n)}$ \cite{Greville1966Note}. This is possible if the number of grid points $N+n$ is less than or equal to the number of parameters $\theta$. In this case, the Gauss-Newton method that we proposed for the variational problem (\ref{varIter}) is identical to the Gauss-Newton method for the L2 minimization (\ref{colloIter}). More specifically, we have
\begin{align}
(\boldsymbol{J} (\theta))^{\dagger} \cdot \nabla_{\theta} L(\theta) = (\boldsymbol{G} \cdot \boldsymbol{\operatorname{JF}}(\theta) )^{\dagger}\cdot \boldsymbol{G}\cdot \boldsymbol{F}(\boldsymbol{x},\theta)  = (\boldsymbol{\operatorname{JF}}(\theta))^{\dagger}\cdot (\boldsymbol{G}^{\dagger} \boldsymbol{G}) \cdot \boldsymbol{F}(\boldsymbol{x},\theta)=(\boldsymbol{\operatorname{JF}}(\theta))^{\dagger}\cdot \boldsymbol{F}(\boldsymbol{x},\theta).
\end{align}}

\section{Semiregular zeros of $\nabla L(\theta)=0$}\label{sec:semiregular}
We will consider the semiregular zeros of $\nabla L(\theta^*)=0$ by the following two definitions. 

\begin{definition}\label{Di-zero} {(Dimension of a Zero \cite{bates2013numerically,zeng2021newton})} Let $\theta^*$ be a zero of a smooth mapping $\nabla L: \Omega \subset \mathbb{R}^m \rightarrow \mathbb{R}^{N+n}$. If there is an open neighborhood $\Omega_z\subset \Omega$ of $\theta_*$ in $\mathbb{R}^m$ such that $\Omega_z\ \cap (\nabla L)^{-1}(\mathbf{0})=\phi(\Lambda)$ where $\mathbf{z} \mapsto \phi(\mathbf{z})$ is a differentiable injective mapping defined in a connected open set $\Lambda$ in $\mathbb{R}^k$ for a certain $k>0$ with $\phi\left(\mathbf{z}_*\right)=\theta_*$ and ${rank}\left(\phi_{\mathbf{z}}\left(\mathbf{z}_*\right)\right)=k$, then the dimension of $\theta_*$ as a zero of $\nabla L$ is defined as
$$
{dim}_{\nabla L}\left(\theta_*\right):={dim}\left({Range}\left(\phi_{\mathbf{z}}\left(\mathbf{z}_*\right)\right)\right) \equiv {rank}\left(\phi_{\mathbf{z}}\left(\mathbf{z}_*\right)\right)=k
$$
\end{definition}

\begin{definition} {(Semiregular Zero \cite{bates2013numerically,zeng2021newton})} A zero $\theta^*\in\mathbb{R}^m$ of a smooth mapping $\theta \mapsto \nabla L(\theta)$ is semiregular if ${dim}_{\nabla L}\left(\theta^*\right)$ is well-defined and identical to {\it nullity} $\left({\mathbf{H}}\left(\theta^*\right)\right)$. Namely
$$
{dim}_{\nabla L}\left(\theta^*\right)+{rank}\left({\mathbf{H}}\left(\theta^*\right)\right)=m.%{\rm the~dimension~of~the~domain~of~}\mathbf {\nabla L}.
$$
\end{definition}

We have the following properties of semiregular zeros in our setup.

%\begin{lemma}\label{perturbation} 
%If $\mathbf{J}\mathbf{v}\neq 0$ and $\|\mathbf{v}\|=1$, for small enough $\epsilon$, then $(\mathbf{J}+\mathbf{Q})\mathbf{v}\neq 0$ and $\|\mathbf{Q}\|\le \epsilon$. And similarly,  for a small enough $\epsilon$ and $\|\mathbf{Q}\|\le \epsilon$, if $(\mathbf{J}+\mathbf{Q})\mathbf{v}\neq 0$, then $\mathbf{J}\mathbf{v}\neq 0$. 
%\end{lemma}
%\begin{proof}
%We prove by contradiction. If $(\mathbf{J}+\mathbf{Q})\mathbf{v}= 0$, then $\mathbf{J}\mathbf{v}=-\mathbf{Q}\mathbf{v}$, now since $\mathbf{J}\mathbf{v}\neq 0$, we assume $\|\mathbf{J}\mathbf{v}\|=\epsilon_1$. Now we choose $\epsilon\le\frac{\epsilon_1}{2}$ and $\|\mathbf{Q}\|\le \epsilon$, then we have $\|\mathbf{J}\mathbf{v}\|=\|\mathbf{Q}\mathbf{v}\|\le \|\mathbf{Q}\|\|\mathbf{v}\|\le\epsilon=\frac{\epsilon_1}{2}$ implying $\epsilon_1\le \frac{\epsilon_1}{2}$ which is a contradiction. The proof of the second part is similar. 
%\end{proof}

\begin{lemma}  { Let $\theta \mapsto \nabla L (\theta)$ be a smooth mapping with a semiregular zero $\theta^*$ and {{$Range(Q(\theta^*))\subseteq Range(J(\theta^*))$}}.  Then there is an open neighborhood $\Omega_*$ of $\theta^*$, such that, for any $\hat{\theta} \in \Omega_*$, the equality $\mathbf{J}(\hat{\theta})^{\dagger} \nabla L (\hat{\theta})=\mathbf{0}$ holds if and only if $\hat{\theta}$ is a semiregular zero of $\nabla L$ in the same branch of $\theta^*$.}
\end{lemma}
\begin{proof}
We follow the proof of Lemma 4 in \cite{zeng2021newton} and note that $|\mathbf{Q}|\leq \epsilon$ for any small $\epsilon$.
First, we claim that there exists a neighborhood $\Omega_1$ of $\theta^*$ such that for every $\hat{\theta} \in \Omega_1$, we have $\nabla L (\hat{\theta})=\mathbf{0}$ and $\mathbf{J}(\hat{\theta})^{\dagger} \nabla L (\hat{\theta})=\mathbf{0}$. Assume that this assertion is false. Then there exists a sequence $\left\{\theta_j\right\}_{j=1}^{\infty}$ converging to $\theta^*$ such that $\mathbf{J}\left(\theta_j\right)^{\dagger} \nabla L \left(\theta_j\right)=\mathbf{0}$ but $\nabla L \left(\theta_j\right) \neq \mathbf{0}$ for all $j=1,2,\ldots$. Let $\mathbf{z} \mapsto \phi(\mathbf{z})$ be the parameterization of the solution branch containing $\theta^*$ as defined in Definition \ref{Di-zero}, with $\phi\left(\mathbf{z}\right)=\theta^*$. From Lemma \ref{lemma2:zeng}, for any sufficiently large $j$, there exists a $\check{\theta}j \in \Omega \cap (\nabla L)^{-1}(\mathbf{0})=\phi(\Delta)$ such that 

\begin{equation}\label{eq:16}
\left\|\theta_j-\check{\theta}_j\right\|_2=\min _{\mathbf{z} \in \Delta}\left\|\theta_j-\phi(\mathbf{z})\right\|_2=\left\|\theta_j-\phi\left(\mathbf{z}_j\right)\right\|_2
\end{equation}
at a certain $\mathbf{z}_j$ with $\phi\left(\mathbf{z}_j\right)=\check{\mathbf{x}}_j$, implying
\begin{equation}\label{eq:17}
\phi_{\mathbf{z}}\left(\mathbf{z}_j\right) \phi_{\mathbf{z}}\left(\mathbf{z}_j\right)^{\dagger} \frac{\theta_j-\phi\left(\mathbf{z}_j\right)}{\left\|\theta_j-\phi\left(\mathbf{z}_j\right)\right\|_2}=\frac{\phi_{\mathbf{z}}\left(\mathbf{z}_j\right)}{\left\|\theta_j-\phi\left(\mathbf{z}_j\right)\right\|_2}\left(\phi_{\mathbf{z}}\left(\mathbf{z}_j\right)^{\dagger}\left(\theta_j-\phi\left(\mathbf{z}_j\right)\right)\right)=\mathbf{0} .
\end{equation}
We claim that $\check{\theta}_j$ converges to $\theta^*$ as $j$ approaches infinity. To see why, assume otherwise. Namely, suppose there exists an $\varepsilon>0$ such that for any $N>0$, there is a $j>N$ with $\left\|\check{\theta}_j-\theta^*\right\|_2 \geq 2 \varepsilon$. However, we know that $\left\|\theta_j-\theta^*\right\|_2<\varepsilon$ for all $j$ larger than some fixed $N$. This implies that
$$
\left\|\check{\theta}_j-\theta_j\right\| \geq\left\|\check{\theta}_j-\theta^*\right\|-\left\|\theta_*-\theta_j\right\|>\varepsilon>\left\|\theta_j-\theta_*\right\|_2
$$
which contradicts \eqref{eq:16}. Therefore, we conclude that $\check{\theta}_j$ converges to $\theta^*$ as $j$ approaches infinity.

Since $\nabla L \left(\theta_j\right) \neq \mathbf{0}$, we have $\theta_j \neq \check{\theta}_j$. Therefore, we can consider the unit vector $\mathbf{v}_j = \left(\theta_j-\check{\theta}j\right) /\left\|\theta_j-\check{\theta}j\right\|_2$ for each $j$. By compactness, there exists a subsequence ${\mathbf{v}_{j_k}}$ that converges to some unit vector $\mathbf{v}$. That is, $\displaystyle\lim_{k \rightarrow \infty} \mathbf{v}_{j_k} = \mathbf{v}$ for some unit vector $\mathbf{v}$.
 Thus
\begin{equation}\label{eq:18}
\begin{aligned}
0 & =\lim _{j \rightarrow \infty} \frac{\mathbf{J}({\theta_j})^{\dagger}\left(\nabla L \left(\check{\theta}_j\right)-\nabla L\left(\theta_j\right)\right)}{\left\|\theta_j-\check{\theta}_j\right\|_2} \\
& =\lim _{j \rightarrow \infty} \frac{\mathbf{J}({\theta_j})^{\dagger} \mathbf{H}\left(\theta_j\right)\left(\check{\theta}_j-\theta_j\right)}{\left\|\check{\theta}_j-\theta_j\right\|_2}=\mathbf{J}({\theta^*})^{\dagger} \mathbf{H}\left(\theta^*\right) \mathbf{v}
\end{aligned}.
\end{equation}
By {the assumption $Range(Q(\theta^*))\subseteq Range(J(\theta^*))$} and noting that $ \mathbf{H}=\mathbf{J}+\mathbf{Q}$ and $\|\mathbf{Q}\|\le \epsilon_2$ for any small $\epsilon_2$, we have $\mathbf{v} \in {Kernel}\left(\mathbf{H}\left(\theta^*\right)\right)$. As a result,
$$
{span}\{\mathbf{v}\} \oplus {Range}\left(\phi_{\mathbf{z}}\left(\mathbf{z}_*\right)\right) \subset {Kernel}
\left(\mathbf{H}\left(\theta^*\right)\right)
$$
since $\mathbf{H}\left(\theta^*\right) \phi_{\mathbf{z}}\left(\mathbf{z}_*\right)=O$ due to $\nabla L (\phi(\mathbf{z})) \equiv \mathbf{0}$ in a neighborhood of $\mathbf{z}_*$. From the limit of \eqref{eq:17} for $j \rightarrow \infty$, the vector $\mathbf{v}$ is orthogonal to ${Range}\left(\phi_{\mathbf{z}}\left(\mathbf{z}_*\right)\right)$ and thus
$$
nullity \left(\mathbf{H}\left(\theta^*\right)\right) \geq {rank}\left(\phi_{\mathbf{z}}\left(\mathbf{z}_*\right)\right)+1
$$
which is a contradiction to the semiregularity of $\theta^*$.
For the special case where $\theta^*$ is an isolated semiregular zero of $\nabla L(\theta)$ with dimension 0, the above proof applies with $\check{\theta}_j=\theta_j$ so that \eqref{eq:18} holds, implying a contradiction to ${nullity}\left(\mathbf{H}\left(\theta^*\right)\right)=0.$

Lemma \ref{lemma3:zeng} implies that there exists a neighborhood $\Omega_2$ of $\theta^*$ such that for every $\theta \in \Omega_2 \cap (\nabla L)^{-1}(\mathbf{0})$, $\theta$ is a semiregular zero of $\nabla L$ in the same branch as $\theta^*$. Therefore, the lemma holds for $\Omega_* = \Omega_1 \cap \Omega_2$.

\end{proof}

{Subsequently, we will provide justification for the semi-regular case within our setup, beginning with the finite element method.}
Let us consider the finite element space and define $V_N$ as the set of functions that can be represented as a linear combination of the basis functions $\phi_i(x)$, where $a_{i}\in \mathbb{R}$ and $i=1,\cdots, m$. Here, $\phi_i(x)$ is a frame defined according to the partition $\mathcal{T}_h$. Thus we can express $V_N$ as
\begin{equation}
V_N = \left\{ \sum_{i=1}^{m} a_{i}\phi_i(x) \right\}.
\end{equation}

{If $\displaystyle u(x,\theta) \in V_N$,} we obtain a linear system $\nabla L(\theta)=A\theta-g$, where $A$ is a square matrix and $\theta=(a_1, a_2, \cdots, a_m)$ represents the coefficients of the basis functions in $V_N$.
\begin{itemize}
    \item {Regular zero:} if $A$ is full rank, $\theta^*$ is unique and therefore an isolated and regular zero. 
    \item {Semiregular zero:} if %$A=A-\nabla_\theta g(\theta^*)$ 
    $A$
    is not full rank, then we denote $rank(A)=r$. We assume that $Kerl (A)=span\{\theta_1, \theta_2, \cdots, \theta_{m-r} \}$ with $\|\theta_i\|=1$ and set
$$
\Delta=\{\theta: \|\theta-\theta^*\|<\delta\}.
$$ 
For any $\theta\in \Delta\cap (\nabla L)^{-1}(0)$, we have 
$$
\theta =\theta^*+\delta_1 \theta_1+\cdots+\delta_{m-r} \theta_{m-r}
$$
with $\|\delta_i\|< \frac{\delta}{m-r}$ and
$
\mathbf{z}=(\delta_1, \delta_2, \cdots, \delta_{m-r}).
$
Then we can construct $\phi$ as $\phi(\mathbf{z})=\theta =\theta^*+\delta_1 \theta_1+\cdots+\delta_{m-r} \theta_{m-r}$ and 
$
\phi_\mathbf{z}(\mathbf{z})=(\theta_1, \theta_2, \cdots,\theta_{m-r})\in \mathcal{R}^{m\times (m-r)}.
$
Moreover, 
$$
{dim}_{\nabla L}\left(\theta^*\right)+{rank}\left(\mathbf{H}\left(\theta^*\right)\right)=rank(\phi_z(\mathbf{z}))+rank(A)=m
$$
which means that $\theta^*$ is a semirrgular zero. 
\end{itemize}

{Next, we proceed to explain the occurrence of semiregular zeros in neural networks with the $ReLU^k$ activation function.}
Let's consider a simple neural network in 1D with domain $\Omega=(-1,1)$:

\begin{equation}
V_N^k = \left\{ \sum_{i=1}^{m} a_{i}\text{ReLU}^k\left(w_{i} x+b_{i}\right), a_{i}\in \mathbb{R}, w_{i} \in \{-1,1\}, b_{i} \in [-1-\delta, 1+\delta]\right\}.
\end{equation}

Because of the scaling property of the ReLU function, we can rewrite $V_N^k$ as
\begin{equation}
V_N^k = \left\{ \sum_{i=1}^{m} a_{i}\text{ReLU}^k\left(x+b_{i}\right), a_{i}\in \mathbb{R}, b_{i} \in [-1-\delta, 1+\delta]\right\}.
\end{equation}

We consider the following semi-regular zero cases:

\begin{itemize}
    \item { Let us consider a simple case where $u=a_1\text{ReLU}^k(x+b_1)$ and $L(\theta)$ is computed with two grid points $x_1$ and $x_1^b$ satisfying $x_1+b_1^*<0$ and $B(x_1^b, \theta^*)=u'(x_1^b,\theta^*)=0$.} Then we can write
$$
\boldsymbol{G}=
\begin{bmatrix} 
w_1\frac{\partial u(x_1, \theta)}{\partial a_1}& w_1^b \frac{\partial u(x_1^b, \theta)}{\partial a_1} \\
w_1\frac{\partial u(x_1, \theta)}{\partial b_1}& w_1^b \frac{\partial u(x_1^b, \theta)}{\partial b_1} \\
\end{bmatrix}
=
\begin{bmatrix} 
0& w_1^b \frac{\partial u(x_1^b, \theta)}{\partial a_1} \\
0& w_1^b \frac{\partial u(x_1^b, \theta)}{\partial b_1} \\
\end{bmatrix}, {\mathbf{JF}=
\begin{bmatrix} 
0&0\\
\frac{\partial B(x_1^b,\theta)}{\partial a_1}& \frac{\partial B(x_1^b,\theta)}{\partial b_1} \\
\end{bmatrix},}
$$
and $$
\nabla_{\theta} L(\theta)=
\begin{bmatrix} 
0& w_1^b \frac{\partial u(x_1^b, \theta)}{\partial a_1} \\
0& w_1^b \frac{\partial u(x_1^b, \theta)}{\partial b_1} \\
\end{bmatrix}
\begin{bmatrix} 
F(x_1, \theta)\\
F(x_1^b, \theta)\\
\end{bmatrix}=w_1^b\begin{bmatrix} 
 \frac{\partial u(x_1^b, \theta)}{\partial a_1}
F(x_1^b, \theta)\\
 \frac{\partial u(x_1^b, \theta)}{\partial b_1} F(x_1^b, \theta)
\\
\end{bmatrix}.
$$
Let us define
$$
\theta=(a_1,b_1)^T, \quad \theta^*=(a_1^*,b_1^*)^T, \quad \theta_1=(0, 1)^T
$$
and
$$
\phi(\mathbf{z})=\theta^*+\delta_1 \theta_1=\theta^*+\delta_1  \theta_1=\theta^*+\delta_1 (0, 1)^T\quad \hbox{with}\quad \mathbf{z}=\delta_1. 
$$
Then, there exists a $\delta$ such that $\Delta=\{\theta: \|\theta-\theta^*\|<\delta\}$, which means $b_1$ is near $b_1^*$.
We have $\Lambda=(-\delta, \delta)$ such that $\phi(\Lambda)=\Delta\cap (\nabla L)^{-1}(0)$ since 
if $\theta^*=(a_1^*,b_1^*)^T$ satisfies $\nabla L(\theta^*)=0$ then for some $\delta$, for any $\theta=(a_1^*, b_1)$ with $\|b_1-b_1^*\|< \delta $ 
satisfies $\nabla L(\theta)=0$, which implies $\Delta\cap (\nabla L)^{-1}(0)=\Delta$. 

We can see that $\phi_z(\mathbf{z})=(0,1)^T$ and $\hbox{rank}(\phi_z(\mathbf{z}))=1$. Furthermore, we observe that 
$$ \boldsymbol {H}(\theta)=
w_1^b\begin{bmatrix} 
\frac{\partial u(x_1^b, \theta)}{\partial a_1}
F_{a_1}(x_1^b, \theta)& \frac{\partial^2 u(x_1^b, \theta)}{\partial a_1\partial b_1}F(x_1^b, \theta)+\frac{\partial u(x_1^b, \theta)}{\partial a_1}F_{b_1}(x_1^b, \theta)\\
 \frac{\partial^2 u(x_1^b, \theta)}{\partial b_1\partial a_1} F(x_1^b, \theta)+ \frac{\partial u(x_1^b, \theta)}{\partial b_1} F_{a_1}(x_1^b, \theta)
&  \frac{\partial^2 u(x_1^b, \theta)}{\partial b_1^2} F(x_1^b, \theta)
+ \frac{\partial u(x_1^b, \theta)}{\partial b_1} F_{b_1}(x_1^b, \theta)
 \\
\end{bmatrix},
$$ 
{ which implies that $\hbox{rank}(\boldsymbol{H}(\theta^*))=1$ due to the fact that $F(x_1^b,\theta^\ast) = 0$.}

Therefore, we have $\hbox{rank}(\phi_z(\mathbf{z}))+\hbox{rank}(\mathbf{H}(\theta^*))=2$, which is the dimension of the domain of $\nabla_\theta L(\theta)$. This implies that in the simple case $u=a_1\hbox{ReLU}(x+b_1)$, $\theta^*$ is a semiregular zero of $\nabla_\theta L(\theta)=0$.

\item 
We next consider the general case of $m>1$ and  assume {$\hbox{rank}(\mathbf{G}(\theta^*))=r< 2m$}. %where  $\hbox{rank}(\mathbf{J}(\theta))=\hbox{rank}(\mathbf{G}(\theta))$. 
By denoting
$$
\theta=(a_1, a_2, \cdots, a_m, b_1, b_2, \cdots, b_m)^T, \quad \theta^*=(a_1^*, a_2^*, \cdots, a_m^*, b_1^*, b_2^*, \cdots, b_m^*),
$$ 
%we assume $b_1^*, b_2^*, \cdots,  b_{2m-r}^*$ do not appear in $\mathbf{G}$ 

{we assume there exists $2m-r$ sample points $x_i$ such that $x_i+b_j^*<0$ for all $j=1,\cdots, m$}

and define 
\begin{align*}
&\theta_1=(0,0,\cdots, 0, 1, 0, \cdots, 0,0, 0,\cdots,0)^T,\\
&\theta_2=(0,0,\cdots, 0, 0, 1, \cdots, 0,0, 0,\cdots,0)^T,\\
&\cdots\\
&\theta_{2m-r}=(0,0,\cdots, 0, 0, 0, \cdots, 0,1, 0,\cdots,0)^T.
\end{align*}
We observe that if $\theta^*=(a_1^*, \cdots, a_m^*, b_1^*, \cdots, b_m^*)$ satisfies $\nabla L(\theta^*)=0$, then $b_i$ is a function of $b_i^*$ and the remaining parameters $a_1^*,\dots,a_m^*$, and hence, if $\theta=(a_1, a_2, \cdots, a_m, b_1, b_2, \cdots, b_m)^T$, then $\nabla L(\theta)=0$ as well. This implies that there exist $\delta_1, \delta_2, \cdots,\delta_{2m-r}$ such that
$$
\phi(\mathbf{z})=\theta^*+\delta_1  \theta_1+\delta_2 \theta_2+\cdots+ \delta_{2m-r}  \theta_{2m-r} 
\hbox{~with~}\mathbf{z}=(\delta_1, \delta_2, \cdots,\delta_{2m-r})
$$
and
$
\Delta=\{\theta: \|\theta-\theta^*\|<\delta\}  
$
such that $\Delta\cap (\nabla L)^{-1}(0)=\Delta$. We define
$$
\Lambda=(-\delta_1, \delta_1)\times (-\delta_2, \delta_2)\times\cdots\times (-\delta_{2m-r}, \delta_{2m-r}),
$$
and have
$$
\phi_z(\mathbf{z})=(\theta_1, \theta_2, \cdots,\theta_{2m-r}).
$$
Clearly, we have $\hbox{rank} (\phi_z(\mathbf{z}))=2m-r$. {If  $F(x_i;\theta^*)=0$ when $x_i+b_j^*>0$ for $i=1,\cdots,N+n$ and $j=1,\cdots, m$, then we have $\hbox{rank} (\mathbf{H}(\theta^*))=r$ which implies that $\hbox{rank} (\phi_z(\mathbf{z}))+\hbox{rank} (\mathbf{J}(\theta^*))=2m-r+r=2m$}. Hence, in this case, $\theta^*$ is a semiregular zero of $\nabla_\theta L(\theta)=0$.

\end{itemize}

\section{Convergence analysis}\label{sec:con}
\subsection{Gauss-Newton method}
In this section, we will analyze the convergence properties of the Gauss-Newton method shown in (\ref{varIter}). We assume that the variational problem has a semiregular zero $\theta^*$ such that $\nabla_{\theta} L(\theta^*) = 0$, and we also assume that $\operatorname{rank}{\boldsymbol{J} (\theta^*)} = r \leq m$. Using singular value decomposition, we can write $J(\theta^*) = U \Sigma V^{T}$, where $J(\theta^*)$ is an $r$-rank semi-positive definite matrix. In particular, we have \cite{wedin1973perturbation} \begin{align}
\Sigma = diag\left( [\sigma_1,\cdots, \sigma_r, 0,\cdots, 0] \right), \text{ with } \sigma_1 \geq \sigma_2 \geq \cdots \geq \sigma_r >0 \text{, } r\leq m.
\end{align}
Thus, the pseudo-inverse can be represented as
\begin{align}
\boldsymbol{J} (\theta^*)^{\dagger} = V \Sigma^{\dagger} U^{T}, \text{ with } \Sigma^{\dagger} = diag\left( \left[\dfrac{1}{\sigma_1},\cdots, \dfrac{1}{\sigma_r}, 0,\cdots, 0\right] \right).
\end{align}
Here, the pseudo-inverse is based on the rank-r projection of $\boldsymbol{J}(\theta)$ denoted as $\boldsymbol{J}_{\text{rank-r}}(\theta)$ in \cite{zeng2021newton}. However, for simplicity, in our paper, we use the term $\boldsymbol{J}(\theta)$ to represent the rank-r projection.

\begin{lemma} \label{lemma6}
Let $\boldsymbol{J} (\theta)$ be the approximated  Hessian defined in (\ref{HessianL}) and $\theta^*$ be the stationary point. Then there exists a small open set $\Omega_{*} \ni \theta^*$ and constants $\epsilon, \zeta, \alpha > 0$ such that for every $\theta_1, \theta_2 \in \Omega_{*}$, the following inequalities holds:
\begin{align}
\Vert \boldsymbol{J} (\theta_2)  \boldsymbol{J} (\theta_2) ^{\dagger} - \boldsymbol{J} (\theta_1)  \boldsymbol{J} (\theta_1) ^{\dagger} \Vert &\leq \zeta \Vert \theta_2 - \theta_1 \Vert, \\
\Vert \nabla_{\theta}L(\theta_2)  - \nabla_{\theta}L(\theta_1)  - \boldsymbol{J} (\theta_1) \left( \theta_2 - \theta_1 \right) \Vert &\leq  \epsilon \Vert \theta_2 - \theta_1 \Vert + \alpha \Vert \theta_2 - \theta_1 \Vert^2.
\end{align}
\end{lemma}
\begin{proof}
From Taylor's expansion, we have 
\begin{align*}
 &\quad \nabla_{\theta}L(\theta_2)  - \nabla_{\theta}L(\theta_1)  - \boldsymbol{J} (\theta_1)  \left( \theta_2 - \theta_1 \right) \\
 &= \boldsymbol{\operatorname{H}}(\theta_1) \left( \theta_2 - \theta_1 \right) - \boldsymbol{J} (\theta_1) \left( \theta_2 - \theta_1 \right) + \mathcal{O}(\Vert \theta_2 - \theta_1 \Vert^2)   \\
&=\boldsymbol{Q} (\theta_1)\left( \theta_2 - \theta_1 \right) + \mathcal{O}(\Vert \theta_2 - \theta_1 \Vert^2).
\end{align*}
Therefore, due to the smoothness of the Hessian with respect to $\theta$ and Lemma~\ref{lemma4}, there exists a small neighborhood $\Omega_*$ of $\theta^*$ and a constant $\alpha > 0$ such that for every $\theta_1, \theta_2 \in \Omega_*$, we have $\Vert \boldsymbol{Q} (\theta_1) \Vert \leq \epsilon$, and the following inequality holds:
\begin{align}
\Vert \nabla_{\theta}L(\theta_2)  - \nabla_{\theta}L(\theta_1)  -\boldsymbol{J} (\theta_1) \left( \theta_2 - \theta_1 \right) \Vert \leq  \epsilon\Vert \theta_2 - \theta_1 \Vert + \alpha \Vert \theta_2 - \theta_1 \Vert^2.
\end{align}
for every $\theta_1, \theta_2 \in \Omega_{*}$. On the other hand, Weyl’s Theorem guarantees the singular value to be continuous with respect to the matrix entries. Therefore, as the open set $\Omega_{*}$ being sufficiently small, it holds
\begin{align*}
&\sup_{\theta \in \Omega_*}\Vert  \boldsymbol{J} (\theta) \Vert  = \sup_{\theta \in \Omega_*} \sigma_1( \boldsymbol{J} (\theta))  \leq C \Vert \boldsymbol{J} (\theta^*)\Vert \\ 
&\sup_{\theta \in \Omega_*}\Vert  \boldsymbol{J} (\theta) ^{\dagger} \Vert  = 
\dfrac{1}{\sup_{\theta \in \Omega_*} \sigma_r( \boldsymbol{J} (\theta)) } \leq C  \Vert \boldsymbol{J} (\theta^*)^{\dagger} \Vert
\end{align*}
for all $\theta \in \Omega_*$. By the smoothness of $\boldsymbol{J} (\theta)$ with respect to $\theta$ and the estimation in Lemma~\ref{lemma5}, we have 
\begin{align*}
&\quad \Vert \boldsymbol{J} (\theta_2)  \boldsymbol{J} (\theta_2) ^{\dagger} - \boldsymbol{J} (\theta_1)  \boldsymbol{J} (\theta_1) ^{\dagger} \Vert  \\ 
& \leq \Vert \boldsymbol{J} (\theta_2)^{\dagger} \Vert \Vert \boldsymbol{J} (\theta_2)  - \boldsymbol{J} (\theta_1) \Vert + \Vert \boldsymbol{J} (\theta_1) \Vert \Vert \boldsymbol{J} (\theta_2) ^{\dagger}  -\boldsymbol{J} (\theta_1) ^{\dagger}   \Vert\\
&\leq \zeta \Vert \theta_2 - \theta_1 \Vert.
\end{align*}
\end{proof}

\begin{theorem}(Convergence Theorem)
Let $L(\theta)$ be a sufficiently smooth target function of $\theta$ and $\boldsymbol{J} (\theta)$ be the approximated Hessian of $L(\theta)$ defined in (\ref{HessianL}). Then for every open neighborhood $\Omega_1$ of $\theta^*$, there exists another  neighborhood $\Omega_2 \ni \theta^*$ such that, from every initial guess $\theta_0 \in \Omega_2$, the sequence $\lbrace \theta_k \rbrace_{k=1}^{\infty}$ generated by the iteration (\ref{varIter}) converges in  $\Omega_1$. Furthermore, $\lbrace \theta_k \rbrace_{k=1}^{\infty}$  has at least linear convergence with a coefficient $\gamma \leq 2\epsilon$ when $k$ is large enough, { where $\epsilon$ is the given constant in Lemma 1}.
\end{theorem}

\begin{proof}
Firstly, let $\Omega_*$ be the small open neighbourhood in Lemma~\ref{lemma6} such that $\Vert \boldsymbol{Q} (\theta) \Vert \leq \epsilon < \frac{1}{2}$. For any open neighbourhood $\Omega_1$ of $\theta^*$, there exists a constant $0<\delta<2$ such that $B(\theta^*, \delta) \subset \Omega_1 \cap \Omega_*$ and for  any $\theta_1, \theta_2 \in B(\theta^*, \delta) $, it holds
\begin{align}\label{assume5.6}
\Vert \boldsymbol{J} (\theta_2)^{\dagger} \Vert \left( \alpha \Vert \theta_2 - \theta_1 \Vert + \zeta \Vert \nabla_{\theta} L (\theta_1)  \Vert \right) \leq h < 1,
\end{align}
{where $\alpha$ and  $\zeta$ are constants introduced in Lemma~\ref{lemma6}}.

On the other hand, there also exists $0<\tau<\frac{\delta}{2}$ such that
\begin{align}\label{assume5.7}
 \Vert \boldsymbol{J} (\theta)^{\dagger} \Vert \Vert\nabla_{\theta}L(\theta)\Vert \leq \dfrac{1-h}{2} \delta  < \dfrac{\delta}{2} <1
\end{align}
holds for any $\theta \in B(\theta^*, \tau)$. 
{ Let $\Omega_2 = B(\theta^*, \tau)$, then by \eqref{assume5.7}, we have for any $\theta_0 \in \Omega_2$, the iteration implies }
\begin{align}
\Vert \theta_1 - \theta^* \Vert \leq \Vert \theta_1 - \theta_0 \Vert + \Vert \theta_0 - \theta^* \Vert \leq  \Vert \boldsymbol{J} (\theta_0)^{\dagger} \Vert \Vert\nabla_{\theta}L(\theta_0)\Vert + \tau < \delta.
\end{align}
{ Next we assume that $\theta_k, \theta_{k-1}\in B(\theta^*, \delta)$ for some integer $k \geq 1$.} From Lemma~\ref{lemma6} we have the following estimation 

\begin{align}
\quad \Vert \theta_{k+1} - \theta_k \Vert &= \Vert \boldsymbol{J} (\theta_k)^{\dagger} \nabla_{\theta}L(\theta_k) \Vert \nonumber\\
& = \Vert \boldsymbol{J} (\theta_k)^{\dagger}\left( \nabla_{\theta}L(\theta_k) - \boldsymbol{J} (\theta_{k-1})  \left(\theta_k - \theta_{k-1} + \boldsymbol{J}(\theta_{k-1})^{\dagger}  \nabla_{\theta}L(\theta_{k-1}) \right) \right)  \Vert \nonumber \\ 
&\leq \Vert  \boldsymbol{J} (\theta_k)^{\dagger} \left( \nabla_{\theta}L(\theta_k)  - \nabla_{\theta}L(\theta_{k-1}) - \boldsymbol{J} (\theta_{k-1}) (\theta_k-\theta_{k-1}) \right) \Vert \nonumber \\ 
&\quad + \Vert \boldsymbol{J} (\theta_k)^{\dagger}  \left(\boldsymbol{J} (\theta_k) \boldsymbol{J} (\theta_k)^{\dagger} - \boldsymbol{J} (\theta_{k-1}) \boldsymbol{J} (\theta_{k-1})^{\dagger}  \right) \nabla_{\theta}L(\theta_{k-1}) \Vert  \nonumber\\ 
&\leq \Vert \boldsymbol{J} (\theta_k)^{\dagger} \Vert \left( \alpha \Vert \theta_k - \theta_{k-1} \Vert + \zeta \Vert \nabla_{\theta} L (\theta_{k-1})  \Vert + \epsilon \right) \Vert \theta_k - \theta_{k-1} \Vert.\label{3.44}
\end{align}
{ Since $\theta_{k}, \theta_{k-1} \in B(\theta^*,\delta)$, then by \eqref{assume5.6}, it leads to }
\begin{align}\label{convergence}
\Vert \theta_{k+1} - \theta_k \Vert < h \Vert \theta_k - \theta_{k-1} \Vert < \Vert \theta_k - \theta_{k-1} \Vert
\end{align}
so the convergence is guaranteed. Therefore, we obtain that $\Vert \theta_{j} - \theta_{j-1} \Vert \leq h^j \Vert \theta_{1} - \theta_{0} \Vert $ for $1\leq j \leq k+1$, and
\begin{align*}
\Vert \theta_{k+1} - \theta^* \Vert &\leq \Vert \theta_0 - \theta^* \Vert + \sum_{j=0}^{k} \Vert  \theta_{k-j+1} - \theta_{k-j} \Vert \nonumber \\
&\leq \Vert \theta_0 - \theta^* \Vert + \sum_{j=0}^{k} h^j \Vert  \theta_{0} - \theta^* \Vert \nonumber\\
&< \dfrac{1}{1-h} \Vert  \theta_{1} - \theta_{0} \Vert +\Vert  \theta_{0} - \theta^* \Vert  \\
 &{< \dfrac{1}{1-h}\dfrac{1-h}{2}\delta + \dfrac{1}{2}\delta  = \delta },
\end{align*}
which completes the induction. Thus we conclude that the sequence $\lbrace \theta_k \rbrace_{k=0}^{\infty} \subset B(\theta^*, \delta) \subset \Omega_1 $ as long as the initial iterate $\theta_0 \in B(\theta^*, \tau) = \Omega_2$.

Secondly, let us define $\hat{\theta} = \lim_{k\rightarrow \infty} \theta_{k}$ so that $\hat{\theta} \in \Omega_1$. By the smoothness of $\nabla_{\theta}L(\theta)$ and the convergence property (\ref{convergence}), there exists a constant $\mu>0$ such that
\begin{align}
\Vert \nabla_{\theta}L(\theta_{k-1}) \Vert &=  \Vert \nabla_{\theta}L(\theta_{k-1}) - \nabla_{\theta}L(\hat{\theta}) \Vert \nonumber \\
&\leq \mu \Vert \theta_{k-1} - \hat{\theta} \Vert \nonumber \\
& \leq \mu\left( \Vert \theta_{k-1} - \theta_k \Vert + \Vert \theta_k - \theta_{k+1} \Vert + \cdots \right) \nonumber \\
& \leq \dfrac{\mu}{1-h} \Vert \theta_k - \theta_{k-1} \Vert.\label{3.53}
\end{align}
Now let us combine (\ref{3.44}) and (\ref{3.53}) to find that
\begin{align}\label{betakk}
\Vert \theta_{k+1} - \theta_k \Vert &\leq \beta \Vert \theta_k - \theta_{k-1} \Vert^2 + \epsilon \Vert \theta_k - \theta_{k-1} \Vert \\
 & = \left( \beta \Vert \theta_k - \theta_{k-1} \Vert + \epsilon \right)\Vert \theta_k - \theta_{k-1} \Vert
\end{align}
for some constant $\beta>0$. 

In the case when $k$ is not large enough, i.e., for $k\le k_0$ with some $k_0$, such that $\beta\Vert \theta_k - \theta_{k-1} \Vert\ge \epsilon$, we have 
\begin{align*}
\Vert \theta_{k+1} - \theta_k \Vert & = \left( \beta \Vert \theta_k - \theta_{k-1} \Vert + \epsilon \right)\Vert \theta_k - \theta_{k-1} \Vert\\
& {\leq \left( \beta \Vert \theta_k - \theta_{k-1} \Vert + \beta\Vert \theta_k - \theta_{k-1} \Vert \right)\Vert \theta_k - \theta_{k-1} \Vert}\\
&=2\beta \Vert \theta_k - \theta_{k-1} \Vert^2\\
&\le (2\beta)^{1+2+4+\cdots+2^{k-1}}   \Vert \theta_1- \theta_0 \Vert^{2^k}\\
&= (2\beta)^{\frac{2^{k}-1}{2}}   \Vert \theta_1- \theta_0 \Vert^{2^k}
\end{align*}

On the other hand, it can be easily seen that 
\begin{align*}
\Vert \theta_{k+1} - \hat{\theta} \Vert &{ \leq \Vert \theta_{k+1} - \theta_{k+2} \Vert + \Vert \theta_{k+2} - \theta_{k+3} \Vert  +  \Vert \theta_{k+3} - \theta_{k+4} \Vert  +  \cdots } \\
&\leq \left( h + h^2 + h^3 + \cdots \right) \Vert \theta_{k+1} - \theta_k \Vert \\
&\leq \dfrac{h}{1-h} \Vert \theta_{k+1} - \theta_k \Vert.
\end{align*}
Therefore,  when $k\le k_0$ with some $k_0$, we have quadratic convergence as follows
\begin{align}\label{quadratic}
\Vert \theta_{k+1} - \hat{\theta} \Vert &\leq \dfrac{h}{1-h} \Vert \theta_{k+1} - \theta_k \Vert\\
&=\dfrac{h}{1-h} (2\beta)^{\frac{2^{k}-1}{2}}   \Vert \theta_1- \theta_0 \Vert^{2^k}.
\end{align}
As $k$ grows sufficiently large with $k\ge \bar{k}$ such that $\beta\Vert \theta_k - \theta_{k-1} \Vert \le \epsilon$, we can use the estimate \eqref{betakk} to obtain
\begin{align*}
\left( \beta \Vert \theta_k - \theta_{k-1} \Vert + \epsilon  \right) \leq  2 \epsilon < 1.
\end{align*}
Therefore 
\begin{align}
\Vert \theta_{k+1} - \theta_k \Vert &\leq 2 \epsilon  \Vert \theta_k - \theta_{k-1} \Vert\\
&\leq (2 \epsilon)^k \Vert \theta^{1} - \theta^{0} \Vert 
%& \leq (2 \epsilon)^k \vert \Omega_* \vert.
\end{align}
Hence 
\begin{align}\label{linear-convegence}
\Vert \theta_{k+1} - \hat{\theta} \Vert \leq \dfrac{h}{1-h} \leq (2 \epsilon)^k \Vert \theta^{1} - \theta^{0} \Vert.
\end{align}

The estimates \eqref{quadratic} and \eqref{linear-convegence} show that $\lbrace \Vert \theta_{k+1} - \hat{\theta} \Vert \rbrace_{k=1}^{\infty}$ is a decreasing sequence that exhibits quadratic convergence when the number of iterations is small, and at least has a linear convergence sequence with coefficient $2\epsilon$ when the number of iterations is large. Since $2\epsilon$ can be very small due to the construction of this iteration, the parameter sequence $\lbrace \theta_k \rbrace_{k=1}^{\infty}$ will converge rapidly in numerical experiments.

\end{proof}

\subsection{Random Gauss-Newton method}
By employing Monte Carlo approximation with random sampling points to approximate the integration of $L(\theta)$, the Gauss-Newton method can be transformed into a random algorithm. Specifically, we define $\xi: \Omega\rightarrow \Gamma$ as a random variable, where $\Gamma$ is a set containing all the combinations of $\tilde{N}+\tilde{n}$ numbers that denote the sampling indices on both the domain and boundaries selected from ${1,2,\cdots, N, N+1, \cdots, N+n},$. The cardinality of $\Gamma$ is given by $|\Gamma|=\begin{pmatrix}N+n\\
m+r\end{pmatrix}$. Here, we assume that the total number of sampling points on the domain and boundary are $N$ and $n$, respectively. In each iteration, we draw $\tilde{N}$ and $\tilde{n}$ random samples from the domain and boundary, respectively.

Therefore, the random Gauss-Newton method can be expressed as follows:
{
\begin{equation}\label{RGauss}
\theta_{k+1}(\xi_k) = \theta_k(\xi_{k}) -  \boldsymbol{J} (\theta_k;\xi_k)^{\dagger} \nabla_{\theta}L(\theta_k;\xi_k). 
\end{equation}
}
Here, $\theta_{k+1}(\xi_k)$ represents the updated parameter vector at the $(k+1)$-th iteration, which depends on the random variable $\xi_k$. 

Additionally, it's important to note that at each step $k$, the sample points on both the domain and boundary are given by:
\[x_{s_1}, x_{s_2}, \cdots,  x_{s_{\tilde{N}}}.\, x_{s_1^b}^b, x_{s_2^b}^b,\cdots, x_{s_{\tilde{n}}^b}^b.\]

%where
%\begin{align}
% \boldsymbol{J}(\theta) =\boldsymbol{G} \cdot \boldsymbol{\operatorname{JF}(\theta)}
% \end{align}
% \hbox{~and~} 
% \begin{align}
%\boldsymbol{G}=  \begin{bmatrix}
%  {w_1} \nabla_{\theta} u(x_{s_1},\theta)  & \cdots %& {w_m} \nabla_{\theta}u(x_{s_m},\theta)  & 
%{w^b_1} \nabla_{\theta}  u(x^b_{s_1^b},\theta)  & \cdots &
%{w^b_n} \nabla_{\theta}  u(x^b_{s_r^b},\theta) 
%  \end{bmatrix}
%\end{align}
%\begin{align}\label{collocationr}
%\boldsymbol{F}(\theta) = 
%\begin{pmatrix}
%-a\Delta u(x_{s_1},\theta) + cu(x_1,\theta) - f(x_1) %& \\
%-a\Delta u(x_{s_1},\theta) + %cu(x_2,\theta) - f(x_2) & \\
%\vdots & \\
%-a\Delta u(x_{s_m},\theta) + cu(x_N,\theta) - f(x_N) %& \\
%\nabla u(x_{s_1^b}^b,\theta) \cdot \boldsymbol{n} & \\
%\nabla u(x_{s_2^b}^b,\theta) \cdot \boldsymbol{n} & %\\
%\vdots & \\
%\nabla u(x_{s_r^b}^b,\theta) \cdot \boldsymbol{n} & \\
%\end{pmatrix},
%\end{align}
%$\boldsymbol{\operatorname{JF}(\theta)}$ is still the Jacobi of 
%$\boldsymbol{F}(\theta)$. 

\begin{lemma} \label{lemma7}
Let $\boldsymbol{J} (\theta)$ be the approximated  Hessian defined in (\ref{HessianL}) and $\theta^*$ be the stationary point. Then there exists a small open set $\Omega_{*} \ni \theta^*$ and constants $\epsilon, \zeta, \alpha > 0$ such that for every $\theta_1, \theta_2 \in \Omega_{*}$, the following inequalities holds:
\begin{align}
\mathbb{E}_{\xi_2,\xi_1}\left(\Vert \boldsymbol{J} (\theta_2,\xi_2)  \boldsymbol{J} (\theta_2,\xi_2) ^{\dagger} - \boldsymbol{J} (\theta_1,\xi_1)  \boldsymbol{J} (\theta_1,\xi_1) ^{\dagger} \Vert\right) &\leq \zeta \mathbb{E}_{\xi}\left(\Vert \theta_2 - \theta_1 \Vert\right), \\
\mathbb{E}_{\xi_2,\xi_1}\left(\Vert \nabla_{\theta}L(\theta_2,\xi_2)  - \nabla_{\theta}L(\theta_1,\xi_1)  - \boldsymbol{J} (\theta_1,\xi_1) \left( \theta_2 - \theta_1 \right) \Vert \right)&\leq  \epsilon \mathbb{E}_{\xi}\left(\Vert \theta_2 - \theta_1 \Vert\right) + \alpha \mathbb{E}_{\xi}\left(\Vert \theta_2 - \theta_1 \Vert\right)^2.
\end{align}
\end{lemma}
\begin{proof}
From Taylor's expansion, we have 
\begin{align*}
 &\quad \nabla_{\theta}L(\theta_2,\xi_2)  - \nabla_{\theta}L(\theta_1,\xi_1)  - \boldsymbol{J} (\theta_1,\xi_1)  \left( \theta_2 - \theta_1 \right) \\
 &= \nabla_{\theta}L(\theta_2,\xi_2) -\nabla_{\theta}L(\theta_1,\xi_2)+\nabla_{\theta}L(\theta_1,\xi_2)-\nabla_{\theta}L(\theta_1,\xi_1)-\boldsymbol{J} (\theta_1,\xi_1)  \left( \theta_2-\theta_1 \right) \\
 &= \boldsymbol{\operatorname{H}}(\theta_1,\xi_2) \left( \theta_2 - \theta_1 \right) +\mathcal{O}(\Vert \theta_2 - \theta_1 \Vert^2)+\nabla_{\theta}L(\theta_1,\xi_2)-\nabla_{\theta}L(\theta_1,\xi_1)\\
 &-\left(\boldsymbol{J} (\theta_1,\xi_1)-\boldsymbol{J} (\theta_1,\xi_2)+\boldsymbol{J} (\theta_1,\xi_2) \right)\left( \theta_2 - \theta_1 \right)  \\
&=\boldsymbol{H} (\theta_1,\xi_2)\left(\theta_2-\theta_1\right)-\boldsymbol{J} (\theta_1,\xi_2)\left(\theta_2-\theta_1\right) + \mathcal{O}(\Vert \theta_2 - \theta_1 \Vert^2)\\
&+\nabla_{\theta}L(\theta_1,\xi_2)-\nabla_{\theta}L(\theta_1,\xi_1)-\left(\boldsymbol{J} (\theta_1,\xi_1)-\boldsymbol{J} (\theta_1,\xi_2)\right)\left( \theta_2 - \theta_1 \right) \\
&=\boldsymbol{Q} (\theta_1,\xi_2)\left(\theta_2-\theta_1\right) + \mathcal{O}(\Vert \theta_2 - \theta_1 \Vert^2)+\nabla_{\theta}L(\theta_1,\xi_2)-\nabla_{\theta}L(\theta_1,\xi_1)-\left(\boldsymbol{J} (\theta_1,\xi_1)-\boldsymbol{J} (\theta_1,\xi_2)\right)\left( \theta_2 - \theta_1 \right).
\end{align*}
Now taking norms and expectation from both sides and noting that 
$$
\mathbb{E}_{\xi_2,\xi_1}\left(\Vert  \nabla_{\theta}L(\theta_1,\xi_2)-\nabla_{\theta}L(\theta_1,\xi_1)\Vert\right)=0
$$
and 
$$
\mathbb{E}_{\xi_2,\xi_1}\left(\Vert \left(\boldsymbol{J} (\theta_1,\xi_1)-\boldsymbol{J} (\theta_1,\xi_2)\right)\left( \theta_2 - \theta_1 \right)\Vert\right)=0
$$
we have 
\begin{align*}
 &\mathbb{E}_{\xi_2,\xi_1}\left(\Vert \nabla_{\theta}L(\theta_2,\xi_2)  - \nabla_{\theta}L(\theta_1,\xi_1)  - \boldsymbol{J} (\theta_1,\xi_1)  \left( \theta_2 - \theta_1 \right)\Vert \right) \\
&\le \mathbb{E}_{\xi_2,\xi_1}\left(\Vert \boldsymbol{Q} (\theta_1,\xi_2)\left(\theta_2-\theta_1\right) \Vert\right)+ \mathbb{E}_{\xi}\left(\mathcal{O}(\Vert \theta_2 - \theta_1 \Vert^2)\right).
\end{align*}
Therefore, due to the smoothness of the Hessian with respect to $\theta$ and Lemma~\ref{lemma4}, there exists a small neighborhood $\Omega_*$ of $\theta^*$ and a constant $\alpha > 0$ such that for every $\theta_1, \theta_2 \in \Omega_*$, we have $\mathbb{E}_{\xi_2}(\Vert \boldsymbol{Q} (\theta_1,\xi_2) \Vert)\leq \epsilon$, and the following inequality holds:
\begin{align}
&\mathbb{E}_{\xi_2,\xi_1}\left(\Vert \nabla_{\theta}L(\theta_2,\xi_2)  - \nabla_{\theta}L(\theta_1,\xi_1)  - \boldsymbol{J} (\theta_1,\xi_1)  \left( \theta_2 - \theta_1 \right)\Vert \right) \\
 &\leq  \epsilon \mathbb{E}_{\xi}(\Vert \theta_2 - \theta_1 \Vert) + \alpha \mathbb{E}_{\xi}(\Vert \theta_2 - \theta_1 \Vert^2)
\end{align}
for every $\theta_1, \theta_2 \in \Omega_{*}$. On the other hand, Weyl’s Theorem guarantees the singular value to be continuous with respect to the matrix entries. Therefore, as the open set $\Omega_{*}$ being sufficiently small, it holds
\begin{align*}
&\sup_{\theta \in \Omega_*}\mathbb{E}_{\xi}(\Vert  \boldsymbol{J} (\theta,\xi) \Vert)  = \sup_{\theta \in \Omega_*} \sigma_1 \mathbb{E}_{\xi}( \boldsymbol{J} (\theta,\xi))  \leq C \mathbb{E}_{\xi}(\Vert \boldsymbol{J} (\theta^*,\xi)\Vert )\\ 
&\sup_{\theta \in \Omega_*}\mathbb{E}_{\xi}(\Vert\boldsymbol{J}(\theta,\xi) ^{\dagger} \Vert)  = 
\dfrac{1}{\sup_{\theta \in \Omega_*} \sigma_r \mathbb{E}_{\xi}( \boldsymbol{J} (\theta,\xi)) } \leq C  \mathbb{E}_{\xi}(\Vert \boldsymbol{J} (\theta^*,\xi)^{\dagger} \Vert)
\end{align*}
for all $\theta \in \Omega_*$. By the smoothness of $\mathbb{E}_{\xi}(\boldsymbol{J} (\theta,\xi))$ with respect to $\theta$ and the estimation in Lemma~\ref{lemma5}, we have 
\begin{align*}
&\quad \mathbb{E}_{\xi_2, \xi_1}( \Vert \boldsymbol{J} (\theta_2,\xi_2)  \boldsymbol{J} (\theta_2,\xi_2) ^{\dagger} - \boldsymbol{J} (\theta_1,\xi_1)  \boldsymbol{J} (\theta_1,\xi_1) ^{\dagger} \Vert) \\ 
& \leq \mathbb{E}_{\xi}(\Vert \boldsymbol{J} (\theta_2, \xi_2)^{\dagger} \Vert)  \mathbb{E}_{\xi_2,\xi_1} (\Vert \boldsymbol{J} (\theta_2,\xi_2)  - \boldsymbol{J} (\theta_1,\xi_1) \Vert) +  \mathbb{E}_{\xi_1}(\Vert \boldsymbol{J} (\theta_1,\xi_1) \Vert)  \mathbb{E}_{\xi_2,\xi_1}(\ \Vert \boldsymbol{J} (\theta_2,\xi_2) ^{\dagger}  -\boldsymbol{J} (\theta_1, \xi_1) ^{\dagger}   \Vert)\\
&\leq \zeta  \mathbb{E}_{\xi} (\Vert \theta_2 - \theta_1 \Vert).
\end{align*}
\end{proof}

\begin{theorem}(Convergence Theorem)
{ Let $L(\theta)$ be a sufficiently smooth target function of $\theta$ and $\boldsymbol{J} (\theta)$ be the approximated Hessian of $L(\theta)$ defined in (\ref{HessianL}). Then for every open neighbourhood $\Omega_1$ of $\theta^*$, there exists another  neighbourhood $\Omega_2 \ni \theta^*$ such that, from every initial guess $\theta_0 \in \Omega_2$, the expectation of the sequence $\lbrace\theta_k\rbrace_{k=1}^{\infty}$ generated by the iteration (\ref{RGauss}),  we have \begin{align}\label{RGN-linear-convegence}
\mathbb{E}(\Vert \theta_{k+1} - \hat{\theta} \Vert) \leq C \epsilon^k,
\end{align}
 where $\hat{\theta}$ is a semi-regular zero and $\epsilon$ is the given constant in Lemma 1}.
\end{theorem}

\begin{proof}
Firstly, Let $\Omega_*$ be the small open neighbourhood in Lemma~\ref{lemma6} such that $ \mathbb{E}_{\xi}(\Vert \boldsymbol{Q} (\theta,\xi) \Vert) \leq \epsilon < \frac{1}{2}$. For any open neighbourhood $\Omega_1$ of $\theta^*$, there exists a constant $0<\delta<2$ such that $B(\theta^*, \delta) \subset \Omega_1 \cap \Omega_*$ and for  any $\theta_1, \theta_2 \in B(\theta^*, \delta) $, it holds
\begin{align}
 \mathbb{E}_{\xi_2}(\Vert \boldsymbol{J} (\theta_2,\xi_2)^{\dagger} \Vert) \left( \alpha  \mathbb{E}_{\xi}(\Vert \theta_2 - \theta_1 \Vert) + \zeta \mathbb{E}_{\xi_1}(\Vert \nabla_{\theta} L (\theta_1,\xi_1)  \Vert )\right) \leq h < 1.
\end{align}
On the other hand, there also exists $0<\tau<\frac{\delta}{2}$ such that
\begin{align}
\mathbb{E}_{\xi}(\Vert \boldsymbol{J} (\theta,\xi)^{\dagger} \Vert) \mathbb{E}_{\xi}(\Vert\nabla_{\theta}L(\theta,\xi)\Vert)\leq \dfrac{1-h}{2} \delta  < \dfrac{\delta}{2} <1
\end{align}
holds for any $\theta \in B(\theta^*, \tau)$. Let $\Omega_2 = B(\theta^*, \tau)$, then for $\forall \theta_0 \in \Omega_2$, the iteration implies 
\begin{align}
\mathbb{E}_{\xi}(\Vert \theta_1 - \theta^* \Vert) \leq \mathbb{E}_{\xi}(\Vert \theta_1 - \theta_0 \Vert) + \mathbb{E}_{\xi}(\Vert \theta_0 - \theta^* \Vert) \leq  \mathbb{E}_{\xi_0}(\Vert \boldsymbol{J} (\theta_0, \xi_0)^{\dagger} \Vert) \mathbb{E}_{\xi_0}(\Vert\nabla_{\theta}L(\theta_0,\xi_0)\Vert) + \tau < \delta.
\end{align}
Next we assume that $\theta_k \in B(\theta^*, \delta)$ for some integer $k \geq 1$. From Lemma~\ref{lemma6} we have the following estimation 
\begin{equation}
\begin{aligned}
&\quad \mathbb{E}_{\xi}(\Vert \theta_{k+1} - \theta_k \Vert)= \mathbb{E}_{\xi}(\Vert \boldsymbol{J} (\theta_k,\xi_k)^{\dagger} \nabla_{\theta}L(\theta_k,\xi_k) \Vert ) \\ %\nonumber\\
& =\mathbb{E}_{\xi_0,\xi_1,\cdots,\xi_{k-1}} \mathbb{E}_{\xi_k}\left(\Vert \boldsymbol{J} (\theta_k,\xi_k)^{\dagger}\left( \nabla_{\theta}L(\theta_k,\xi_k) - \boldsymbol{J} (\theta_{k-1},\xi_{k-1})  \left(\theta_k - \theta_{k-1} 
+ \boldsymbol{J}(\theta_{k-1},\xi_{k-1})^{\dagger}  \nabla_{\theta}L(\theta_{k-1},\xi_{k-1}) \right) \right)\Vert \right)  \\ %\nonumber \\ 
&\leq \mathbb{E}_{\xi_0,\xi_1,\cdots,\xi_{k-1}} \mathbb{E}_{\xi_k}\left(\Vert  \boldsymbol{J} (\theta_k,\xi_k)^{\dagger} \left( \nabla_{\theta}L(\theta_k,\xi_k)  - \nabla_{\theta}L(\theta_{k-1},\xi_{k-1}) - \boldsymbol{J} (\theta_{k-1},\xi_{k-1}) (\theta_k-\theta_{k-1}) \right) \Vert\right) \\ %\nonumber \\ 
&\quad + \mathbb{E}_{\xi_0,\xi_1,\cdots,\xi_{k-1}} \mathbb{E}_{\xi_k}\left(\Vert \boldsymbol{J} (\theta_k,\xi_k)^{\dagger}  \left(\boldsymbol{J} (\theta_k,\xi_k) \boldsymbol{J} (\theta_k,\xi_k)^{\dagger} - \boldsymbol{J} (\theta_{k-1},\xi_{k-1}) \boldsymbol{J} (\theta_{k-1},\xi_{k-1})^{\dagger} \right) \nabla_{\theta}L(\theta_{k-1},\xi_{k-1}) \Vert \right) \\ %\nonumber \\ 
&\leq  \mathbb{E}_{\xi_k}( \Vert \boldsymbol{J} (\theta_k,\xi_k)^{\dagger} \Vert) \left( \alpha \mathbb{E}_{\xi}(\Vert \theta_k - \theta_{k-1} \Vert) + \zeta \mathbb{E}_{\xi_{k-1}}(\Vert \nabla_{\theta} L (\theta_{k-1},\xi_{k-1})  \Vert) + \epsilon \right)  \mathbb{E}_{\xi}(\Vert \theta_k - \theta_{k-1} \Vert).\label{RGN-3.44}
\end{aligned}
\end{equation}
Since $\theta_{k} \in B(\theta^*,\delta)$, then it leads to 
\begin{align}\label{RGN-convergence}
\mathbb{E}_{\xi}( \Vert \theta_{k+1} - \theta_k \Vert) < h \mathbb{E}_{\xi}( \Vert \theta_k - \theta_{k-1} \Vert) < \mathbb{E}_{\xi}(\Vert \theta_k - \theta_{k-1} \Vert)
\end{align}
so the convergence is guaranteed. Therefore, we obtain that $\mathbb{E}_{\xi}(\Vert \theta_{j} - \theta_{j-1} \Vert) \leq h^j \mathbb{E}_{\xi}(\Vert \theta_{1} - \theta_{0} \Vert)$ for $1\leq j \leq k+1$, and
\begin{align*}
\mathbb{E}_{\xi}(\Vert \theta_{k+1} - \theta^* \Vert) &\leq\mathbb{E}_{\xi}( \Vert \theta_0 - \theta^* \Vert) + \sum_{j=0}^{k} \mathbb{E}_{\xi}(\Vert  \theta_{k-j+1} - \theta_{k-j} \Vert )\nonumber \\
&\leq \mathbb{E}_{\xi}(\Vert \theta_0 - \theta^* \Vert) + \sum_{j=0}^{k} h^j \mathbb{E}_{\xi}(\Vert  \theta_{0} - \theta^* \Vert)\nonumber\\
&< \dfrac{1}{1-h} \mathbb{E}_{\xi}(\Vert  \theta_{1} - \theta_{0} \Vert )+\mathbb{E}_{\xi}(\Vert  \theta_{0} - \theta^* \Vert ) \\
&< \dfrac{1}{1-h}\dfrac{h-1}{2}\delta + \dfrac{1}{2}\delta  =  \delta,
\end{align*}
which completes the induction. Thus we conclude that the sequence $\lbrace \mathbb{E}_{\xi}(\theta_k) \rbrace_{k=0}^{\infty} \subset B(\theta^*, \delta) \subset \Omega_1 $ as long as the initial iterate $\mathbb{E}_{\xi}(\theta_0) \in B(\theta^*, \tau) = \Omega_2$.

Secondly, let us define $\mathbb{E}_{\xi}(\hat{\theta})= \lim_{k\rightarrow \infty} \mathbb{E}_{\xi}(\theta_{k})$ so that $\mathbb{E}_{\xi}(\hat{\theta}) \in \Omega_1$. By the smoothness of $\mathbb{E}_{\xi}(\nabla_{\theta}L(\theta,\xi))$ and the convergence property (\ref{RGN-convergence}), there exists a constant $\mu>0$ such that
\begin{align}
\mathbb{E}_{\xi_{k-1}}(\Vert \nabla_{\theta}L(\theta_{k-1},\xi_{k-1}) \Vert) &= \mathbb{E}_{\xi_{k-1}}(\Vert \nabla_{\theta}L(\theta_{k-1},\xi_{k-1}) - \nabla_{\theta}L(\hat{\theta},\xi_{k-1}) \Vert) \nonumber \\
&\leq \mu \mathbb{E}_{\xi}(\Vert \theta_{k-1} - \hat{\theta} \Vert) \nonumber \\
& \leq \mu \mathbb{E}_{\xi}(\left( \Vert \theta_{k-1} - \theta_k \Vert) + \mathbb{E}_{\xi}(\Vert \theta_k - \theta_{k+1} \Vert )+ \cdots \right) \nonumber \\
& \leq \dfrac{\mu}{1-h} \mathbb{E}_{\xi}(\Vert \theta_k - \theta_{k-1} \Vert).\label{RGN-3.53}
\end{align}
Now let us combine (\ref{RGN-3.44}) and (\ref{RGN-3.53}) to find that
\begin{align}\label{RGN-betakk}
\mathbb{E}_{\xi}(\Vert \theta_{k+1} - \theta_k \Vert) &\leq \beta \mathbb{E}_{\xi}(\Vert \theta_k - \theta_{k-1} \Vert)^2 + \epsilon \mathbb{E}_{\xi}(\Vert \theta_k - \theta_{k-1} \Vert) \\
 & = \left( \beta \mathbb{E}_{\xi}(\Vert \theta_k - \theta_{k-1} \Vert)+ \epsilon \right)\mathbb{E}_{\xi}(\Vert \theta_k - \theta_{k-1} \Vert)
\end{align}
for some constant $\beta>0$. 

In the case when $k$ is not large enough, i.e., for $k\le k_0$ with some $k_0$, such that $\beta\Vert \theta_k - \theta_{k-1} \Vert\ge \epsilon$, we have 
\begin{align*}
\mathbb{E}_{\xi}(\Vert \theta_{k+1} - \theta_k \Vert )& = \left( \beta  \mathbb{E}_{\xi}(\Vert \theta_k - \theta_{k-1} \Vert) + \epsilon \right) \mathbb{E}_{\xi}(\Vert \theta_k - \theta_{k-1} \Vert)\\
& = \left( \beta \mathbb{E}_{\xi}(\Vert \theta_k - \theta_{k-1} \Vert) + \beta\mathbb{E}_{\xi}( \Vert \theta_k - \theta_{k-1} \Vert )\right)\mathbb{E}_{\xi}( \Vert \theta_k - \theta_{k-1} \Vert)\\
&=2\beta \mathbb{E}_{\xi}(\Vert \theta_k - \theta_{k-1} \Vert)^2\\
&\le (2\beta)^{1+2+4+\cdots+2^{k-1}}   \mathbb{E}_{\xi}(\Vert \theta_1- \theta_0 \Vert)^{2^k}\\
&= (2\beta)^{\frac{2^{k}-1}{2}}   \mathbb{E}_{\xi}(\Vert \theta_1- \theta_0 \Vert)^{2^k}
\end{align*}

On the other hand, it can be easily seen that 
\begin{align*}
\mathbb{E}_{\xi}(\Vert \theta_{k+1} - \hat{\theta} \Vert) &= \mathbb{E}_{\xi}( \Vert \theta_{k+1} - \theta^{k+2} \Vert) + \mathbb{E}_{\xi}(\Vert \theta^{k+2} - \theta^{k+3} \Vert ) +  \mathbb{E}_{\xi}(\Vert \theta^{k+3} - \theta^{k+4} \Vert)  +  \cdots \\
&\leq \left( h + h^2 + h^3 + \cdots \right)\mathbb{E}_{\xi}( \Vert \theta_{k+1} - \theta_k \Vert) \\
&\leq \dfrac{h}{1-h} \mathbb{E}_{\xi}(\Vert \theta_{k+1} - \theta_k \Vert).
\end{align*}
Therefore,  when $k\le k_0$ with some $k_0$, we have quadratic convergence as follows
\begin{align}\label{RGN-quadratic}
\mathbb{E}_{\xi}(\Vert \theta_{k+1} - \hat{\theta} \Vert )&\leq \dfrac{h}{1-h} \mathbb{E}_{\xi}(\Vert \theta_{k+1} - \theta_k \Vert)\\
&=\dfrac{h}{1-h} (2\beta)^{\frac{2^{k}-1}{2}}  \mathbb{E}_{\xi}(\Vert \theta_1- \theta_0 \Vert)^{2^k}.
\end{align}
As $k$ grows sufficiently large with $k\ge \bar{k}$ such that $\beta\mathbb{E}_{\xi}( \Vert \theta_k - \theta_{k-1} \Vert) \le \epsilon$, we can use the estimate \eqref{RGN-betakk} to obtain
\begin{align*}
\left( \beta \mathbb{E}_{\xi}(\Vert \theta_k - \theta_{k-1} \Vert) + \epsilon  \right) \leq  2 \epsilon < 1.
\end{align*}
Therefore 
\begin{align}
\mathbb{E}_{\xi}(\Vert \theta_{k+1} - \theta_k \Vert )&\leq 2 \epsilon \mathbb{E}_{\xi}( \Vert \theta_k - \theta_{k-1} \Vert)\\
&\leq (2 \epsilon)^k \mathbb{E}_{\xi}(\Vert \theta^{1} - \theta^{0} \Vert )
%& \leq (2 \epsilon)^k \vert \Omega_* \vert.
\end{align}
Hence 
\begin{align}\label{RGN-linear-convegence}
\mathbb{E}_{\xi}(\Vert \theta_{k+1} - \hat{\theta} \Vert) \leq \dfrac{h}{1-h} \leq (2 \epsilon)^k \mathbb{E}_{\xi}(\Vert \theta^{1} - \theta^{0} \Vert).
\end{align}
\end{proof}
The estimates \eqref{RGN-quadratic} and \eqref{RGN-linear-convegence} show that $\lbrace \mathbb{E}_{\xi}(\Vert \theta_{k+1} - \hat{\theta} \Vert) \rbrace_{k=1}^{\infty}$ is a decreasing sequence that exhibits quadratic convergence when the number of iterations is small, and at least has a linear convergence sequence with coefficient $2\epsilon$ when the number of iterations is large. Since $2\epsilon$ can be very small due to the construction of this iteration, the parameter sequence $\lbrace \mathbb{E}_{\xi}(\theta_k) \rbrace_{k=1}^{\infty}$ will converge rapidly in numerical experiments.

{\section{Numerical experiments}\label{sec:num}
In this section, we numerically investigate the efficiency and robustness of the Gauss-Newton method by solving the model problem \eqref{minEnergy} in different dimensions. All experiments in the following were run on a single NVIDIA TESLA P100 GPU with double precision.  All the code in this section can be found on GitHub: \url{https://github.com/Jinxl-pp/GaussNewtonDRM}.

\begin{example}\label{example:1}
A second-order elliptic equation in one-dimensional space is given as follows:
\begin{align*}
    -&u^{\prime \prime}(x) + u(x) = f(x),  ~x \in (-1,1),\\
        &u^{\prime}(-1) = u^{\prime}(1) = 0.
\end{align*}
\end{example}
The corresponding variational problem becomes $$ \min_{v\in H^1}\mathcal{J}(v) = \int_{-1}^{1} \left(\frac{1}{2}(v^{\prime})^2 + \frac{1}{2}v^2 - fv \right) \mathrm{d}x.$$ By choosing $f(x)=(\pi^2+1)\cos(\pi x)$, we have  $u(x) = \cos(\pi x)$. We employ shallow neural networks represented as $u_n$ with a ReLU$^3$ activation function. Then, for the discretized energy denoted as $\mathcal{J}(u_n)$, we utilize a piecewise Gauss-Legendre quadrature rule with 12,000 sample points. This rule is based on the 2-point Gauss-Legendre rule and is distributed across 6,000 uniform sub-intervals within the range of $[-1,1]$. These sample points serve as the training data. As for the testing dataset, we extract 16,000 sample points from the same quadrature rule across 8,000 uniform sub-intervals. These testing points are used for calculating the numerical errors.

Now we proceed to compare the Gauss-Newton method with other training techniques, such as SGD, ADAM, and L-BFGS, across various neural network architectures. The adaptive learning rate is implemented through a back-tracking strategy. This adjustment is crucial as the Gauss-Newton method exhibits local convergence, necessitating global optimization techniques for the warming-up stage and achieving global convergence.

During training with Newton-type algorithms (L-BFGS, Gauss-Newton), we conduct 1000 epochs, while for gradient-based algorithms (SGD and ADAM), we extend the training to 20,000 epochs. Figure~\ref{fig:loss_1d} illustrates the detailed dynamics of the training loss in both $L^2$ and $H^1$. The shaded areas represent the range between the first and third quartiles of the 10 experiments, while the lines depict the median values. Since $\mathcal{J}(v) - \mathcal{J}(u) = \frac{1}{2}\|v-u\|^2_{H^1}$, greater accuracy in the $H^1$ norm signifies better convergence of the loss function.

Table \ref{tab:error_1d_compare} presents the best testing error from 10 experiments with independently initialized network parameters, while the architecture remains fixed in each column. It highlights that the accuracy of the Gauss-Newton method is at least two orders of magnitude higher than that of gradient-based training algorithms and one order higher than L-BFGS.
\begin{table}[H]
\centering
\begin{tabular}{cc|c|c|c|c|c}
\multicolumn{2}{c|}{Hidden Layer Width}                          & 16      & 32      & 64      & 128     & 256     \\ \hline
\multicolumn{1}{c|}{\multirow{2}{*}{SGD}}          & $L^2$-error & 7.04e-3 & 4.49e-3 & 3.53e-3 & 3.95e-3 & 3.98e-3 \\
\multicolumn{1}{c|}{}                              & $H^1$-error & 6.98e-2 & 5.45e-2 & 4.26e-2 & 4.57e-2 & 4.73e-2 \\ \hline
\multicolumn{1}{c|}{\multirow{2}{*}{ADAM}}         & $L^2$-error & 1.89e-3 & 1.22e-3 & 2.24e-4 & 2.03e-4 & 1.37e-4 \\
\multicolumn{1}{c|}{}                              & $H^1$-error & 2.76e-2 & 1.97e-2 & 6.07e-3 & 4.76e-3 & 4.01e-3 \\ \hline
\multicolumn{1}{c|}{\multirow{2}{*}{L-BFGS}}       & $L^2$-error & 2.81e-4 & 1.24e-4 & 4.19e-5 & 2.18e-5 & 2.36e-5 \\
\multicolumn{1}{c|}{}                              & $H^1$-error & 6.96e-3 & 3.55e-3 & 1.53e-3 & 8.19e-4 & 9.06e-4 \\ \hline
\multicolumn{1}{c|}{\multirow{2}{*}{Gauss-Newton}} & $L^2$-error & \textbf{7.86e-5} & \textbf{3.45e-5} & \textbf{5.83e-6} & \textbf{2.71e-6} & \textbf{3.78e-7} \\
\multicolumn{1}{c|}{}                              & $H^1$-error & \textbf{2.43e-3} & \textbf{1.30e-3} & \textbf{3.71e-4} & \textbf{2.05e-4} & \textbf{4.50e-5} \\ \hline
\end{tabular}
\caption{Errors in both $L^2$ and $H^1$ norms for various training algorithms with different neural network architectures.}
\label{tab:error_1d_compare}
\end{table}

Additionally, we recorded the total training time required to achieve different tolerances in $L^2$-error for different methods in Table~\ref{tab:error_1d_time}, utilizing a hidden layer with 64 nodes. The results demonstrate that the Gauss-Newton method is more efficient, particularly when higher-order accuracy is essential.

\begin{table}[H]
\centering
\begin{tabular}{c|c|c|c|c}
Stopping Criterion          & SGD     & ADAM    & L-BFGS & Gauss-Newton \\ \hline
${L^2}\text{-error} < 1e-2$ & \textbf{2.15s}   & 3.15s   & 3.70s  & 17.25s                            \\
${L^2}\text{-error} < 1e-3$ & -       & 39.76s  & \textbf{7.54s}  & 18.20s                            \\
${L^2}\text{-error} < 1e-4$ & -       & -       & 83.14s & \textbf{58.07s }                           \\
${L^2}\text{-error} < 1e-5$ & -       & -       & -      & \textbf{132.72s}                           \\ \hline
Average Iteration Time:     & \textbf{0.0029s} & 0.0034s & 0.72s  & 0.17s  \\ \hline                          
\end{tabular}
\caption{Training time for different training algorithms with a hidden layer width of 64. The average iteration time is computed by taking the average of 1000 iterations for each training algorithm.}
\label{tab:error_1d_time}
\end{table}

\begin{figure}[H]
	\centering
	\subfigure[$L^2$-error of Newton-type algorithms]
	{ \includegraphics[width=0.45\textwidth]{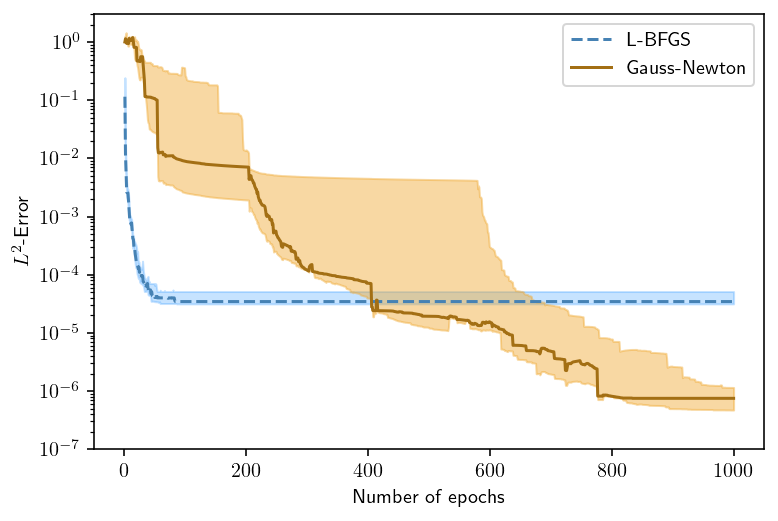}}
		\subfigure[$L^2$-error of gradient-based algorithms]
	{ \includegraphics[width=0.45\textwidth]{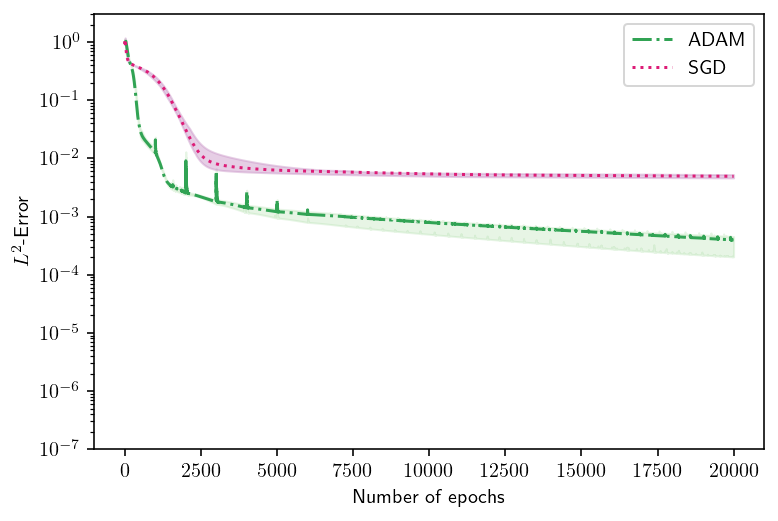}}\\
        \subfigure[$H^1$-error of Newton-type algorithms]
	{ \includegraphics[width=0.45\textwidth]{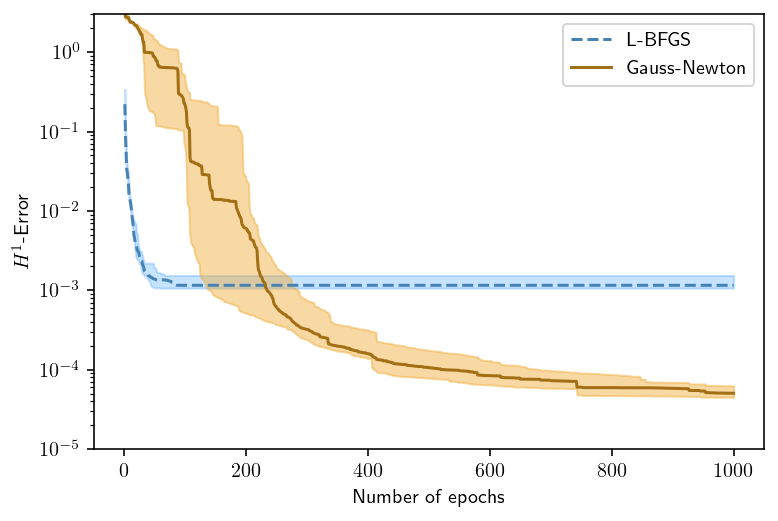}}
        \subfigure[$H^1$-error of gradient-based algorithms]
	{ \includegraphics[width=0.45\textwidth]{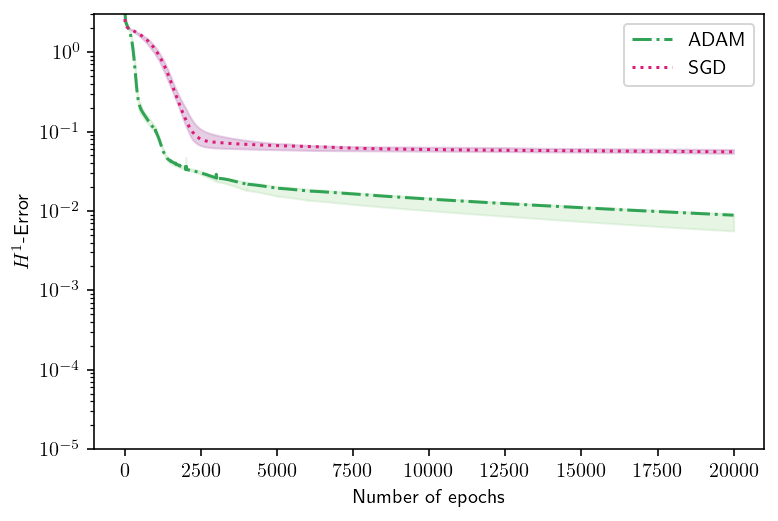}}
	\caption{Testing Errors vs. Iterations. \textbf{Left:} Newton-type training algorithms. The Gauss-Newton method is employed with a back-tracking strategy to determine the learning rate, while L-BFGS is applied with the strong Wolfe condition. \textbf{Right:} Gradient-based training algorithms are applied with an initial learning rate of $1 \times 10^{-3}$, which is then halved every 1000 epochs until it reaches $1 \times 10^{-5}$.
 }\label{fig:loss_1d}
\end{figure}

\begin{example}\label{example:2}
    We consider the following elliptic equation in two-dimensional space:
    \begin{align*}
        -&\Delta u(x,y)+u(x,y) = f(x,y), ~(x,y)\in \Omega = (0,1)^2,\\
         &\frac{\partial u(x,y)}{\partial n} = 0, ~(x,y) \in \partial (0,1)^2.
    \end{align*}
\end{example}
The corresponding variational problem for the 2D equation is given as 
\begin{equation*}
    \min_{v\in H^1} \mathcal{J}(v) = \int_{\Omega} \left(\dfrac{1}{2} \nabla v \cdot \nabla v + \dfrac{1}{2}v^2 - fv \right)\mathrm{d}x\mathrm{d}y
\end{equation*}
By choosing an appropriate source term, we have an exact solution $u(x,y) = \cos(2\pi x) \cos(2\pi y)$. In this example, we employ deep neural networks represented as $u_n$ with ReLU$^3$ activation function and an architecture of 2-20-20-1, which contains two hidden layers with 20 neurons in each. Therefore, the total number of parameters in this architecture is 501. Similarly, we discretize the energy $\mathcal{J}(u_n)$ using a piecewise Gauss-Legendre rule. We take 160000 training samples based on the tensored 2-point Gauss-Legendre rule across $200\times 200$ uniform sub-rectangles in $[0,1]^2$. For the testing dataset used for calculating the numerical errors, we extract 360000 samples using the same quadrature rule across $300 \times 300$ uniform sub-rectangles. The comparison of the four different training algorithms is listed in Table~\ref{tab:error_2d_compare}, which shows the best approximation results of 10 independent experiments with different neural network initializations. For the training with Newton-type algorithms (L-BFGS, Gauss-Newton) we conduct 2500 epochs, and for the gradient-based algorithms (SGD, ADAM) we extend the training to 20000 epochs. Further, the training details of the 10 experiments are plotted in Figure~\ref{fig:loss_2d}, which illustrates the training dynamics in both  $L^2$ and $H^1$ norms. The shaded area represents the range between the first and third quartiles of these experiments, and lines  depict the median values. In both Table~\ref{tab:error_2d_compare} and Figure~\ref{fig:loss_2d} it shows that the accuracy of the Gauss-Newton method is at least two orders of magnitude higher than the gradient-based algorithms and one order higher than L-BFGS. In addition, it also shows the robustness of the Gauss-Newton algorithm under different dimensions while other training algorithms appear simultaneously more oscillations in the training process when compared with their one-dimensional training dynamics in Figure~\ref{fig:loss_1d}.

\begin{table}[H]
\centering
\begin{tabular}{c|c|c|c|c}
Numerical Error   & SGD     & ADAM    & L-BFGS & \multicolumn{1}{c|}{Gauss-Newton} \\ \hline
${L^2}\text{-error}$ & 1.13e-2 & 1.73e-3 & 4.77e-4 & \textbf{2.88e-5}                           \\
${H^1}\text{-error}$ & 2.02e-1 & 5.25e-2 & 1.77e-2 & \textbf{1.57e-3}                          \\ \hline
\end{tabular}
\caption{Errors in both $L^2$ and $H^1$ norms for various training algorithms with a deep architecture 2-20-20-1.}\label{tab:error_2d_compare}
\end{table}

%\begin{table}[H]
%\centering
%\begin{tabular}{c|c|c|c|c}
%Stopping Criterion          & SGD     & ADAM    & L-BFGS & Gauss-Newton \\ \hline
%${L^2}\text{-error} < 1e-1$ & 64.37s  & \textbf{9.77s}   & 174.68s   & 1.07h               %         \\
%${L^2}\text{-error} < 1e-2$ & -       & \textbf{23.18s}  & 420.52s  & 1.17h                         \\
%${L^2}\text{-error} < 1e-3$ & -       & -       & \textbf{0.50h} & 1.32h                                   \\
%${L^2}\text{-error} < 1e-4$ & -       & -       &  -      & \textbf{1.75h}                                    \\ \hline
%Average Iteration Time:     & \textbf{0.015s}  & 0.016s  & 6.29s  &  29.01s      \\ \hline  %\end{tabular}
%\caption{Training time for different training algorithms with a deep architecture 2-20-20-1. The average iteration time is computed by taking the average of 1000 iterations for each training algorithm.}
%\label{tab:error_2d_time}
%\end{table}

\begin{figure}[H]
	\centering
	\subfigure[$L^2$-error of Newton-type algorithms]
	{ \includegraphics[width=0.45\textwidth]{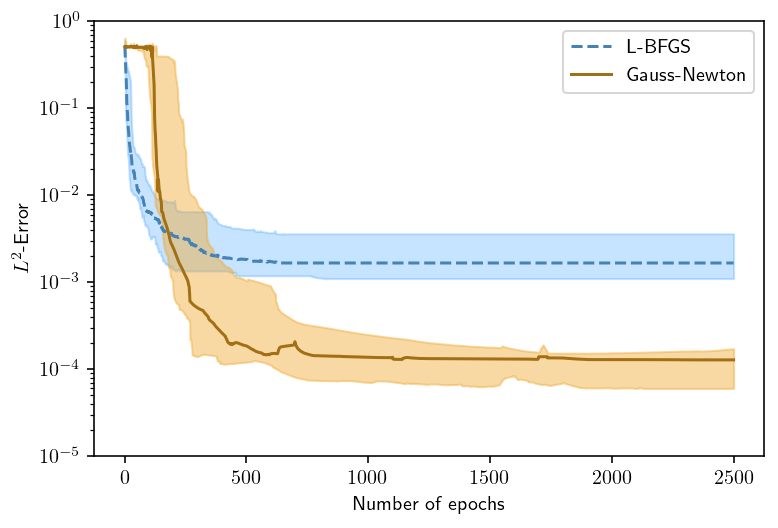}}
		\subfigure[$L^2$-error of gradient-based algorithms]
	{ \includegraphics[width=0.45\textwidth]{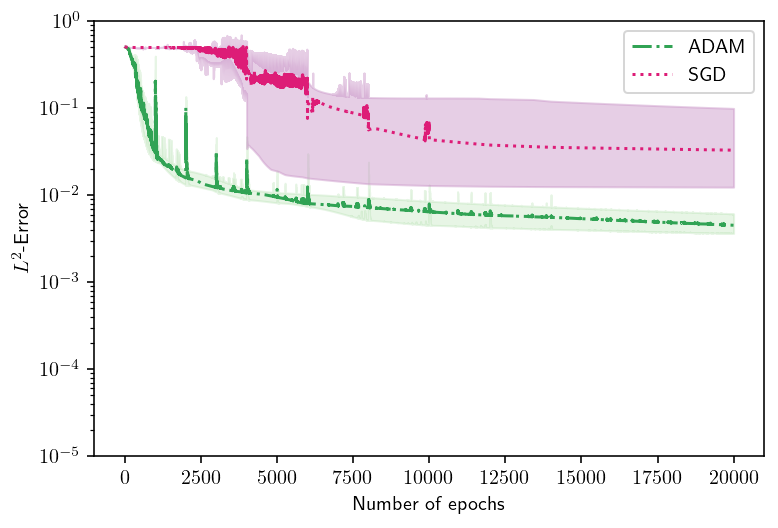}}\\
        \subfigure[$H^1$-error of Newton-type algorithms]
	{ \includegraphics[width=0.45\textwidth]{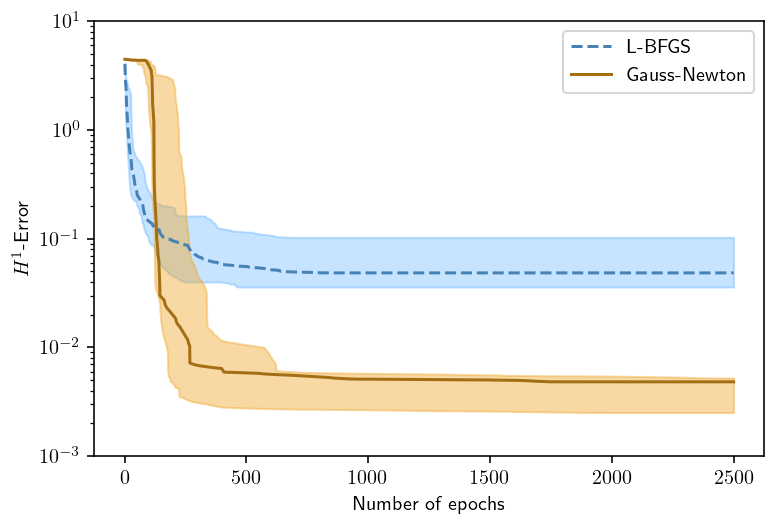}}
        \subfigure[$H^1$-error of gradient-based algorithms]
	{ \includegraphics[width=0.45\textwidth]{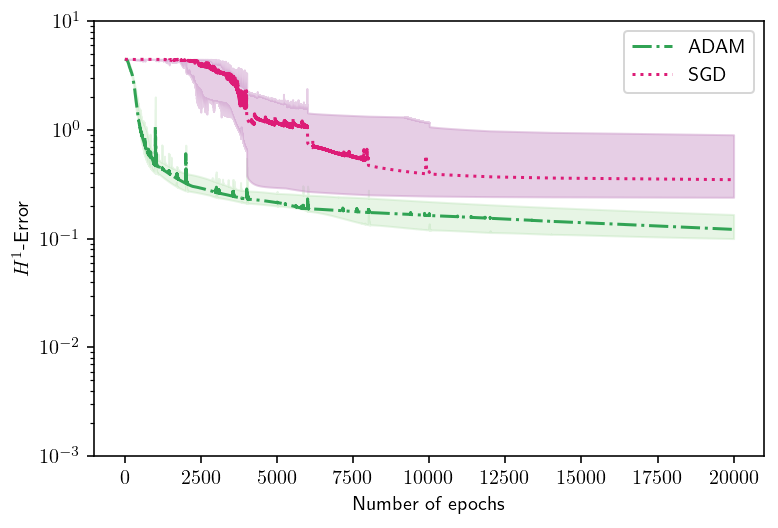}}
	\caption{Testing Errors vs. Iterations. \textbf{Left:} Newton-type training algorithms. The Gauss-Newton method is employed with a back-tracking strategy to determine the learning rate, while L-BFGS is applied with the strong Wolfe condition. \textbf{Right:} Gradient-based algorithms. ADAM is applied with an initial learning rate of $1 \times 10^{-3}$, which is then halved every 2000 epochs until it reaches $1 \times 10^{-5}$. SGD is applied with an initial learning rate of $1\times 10^{-2}$ which is halved every 2000 epochs before going below $1\times 10^{-4}$.
 }\label{fig:loss_2d}
\end{figure}

\begin{example}[High dimensional example]
    Let us denote $x:=(x_1,x_2,\cdots,x_5)$. A second-order elliptic equation in five-dimensional space is considered as follows:
    \begin{align*}
    -&\Delta u(x) + \pi^2 u(x) = 2\pi^2\sum_{k=1}^{5}\cos(\pi x_k),  ~x \in (0,1)^5,\\
        &\dfrac{\partial u(x)}{\partial n} = 0, ~x \in \partial (0,1)^5.
    \end{align*}
    Then the corresponding variational problem becomes 
    \begin{equation*}
        \min_{v\in H^1} \mathcal{J}(v) = \int_{\Omega}\left( \dfrac{1}{2}\nabla v \cdot \nabla v + \dfrac{\pi^2}{2} v^2 - 2\pi^2 v \sum_{k=1}^{5}\cos(\pi x_k) \right)\mathrm{d}x.
    \end{equation*}
\end{example}

The exact solution of the equation is given by $u(x) = \sum_{k=1}^{5}\cos(\pi x_k)$. In this example we employ a shallow neural network $u_n$,  with 64 neurons in the hidden layer and with ReLU$^4$ as its activation function, to solve the equation. The total number of parameters in this architecture is 448. In order to obtain a relatively good accuracy when computing integrals in five-dimensional spaces, we generate 16000 training samples and 20000 testing samples in $[0,1]^5$ using the quasi-Monte-Carlo method associated with the Halton sequence. The training samples will only be used for optimization and the testing samples will be used for computing the relative numerical errors in both $L^2$ and $H^1$ norms. Next, we give a comparison between the Gauss-Newton method and other popular training methods. The best approximation results of 10 independent experiments are listed in Table~\eqref{tab:error_5d_compare}. We conduct 5000 epochs when training the Newton-type algorithms (L-BFGS, Gauss-Newton) and for the gradient-based method (SGD, ADAM) the training epochs are extended to 20000. 

\begin{table}[H]
\centering
\begin{tabular}{c|c|c|c|c}
Numerical Error   & SGD     & ADAM    & L-BFGS & \multicolumn{1}{c|}{Gauss-Newton} \\ \hline
Relative ${L^2}\text{-error}$ & 1.74e-1 & 5.41e-3 & 5.73e-3 & \textbf{5.09e-3}                           \\
Relative ${H^1}\text{-error}$ & 3.42e-1 & 2.08e-2 & 2.16e-2 & \textbf{1.44e-2}                          \\ \hline
\end{tabular}
\caption{Errors in both $L^2$ and $H^1$ norms for various training algorithms with a shallow neural network of 64 neurons.}\label{tab:error_5d_compare}
\end{table}

Moreover, the training details of these 10 experiments are shown in Figure~\ref{fig:loss_5d} below, which gives the training dynamics in the sense of both $L^2$ and $H^1$ norm. Again, the line represents the median values of the numerical relative errors and the shaded area that covers the line represents the range between the first and the third quartiles. Due to the limit of the computational resource and the insufficient accuracy of using the quasi-Monte-Carlo method instead of another high-resolution quadrature rule, all four training algorithms showed less approximation accuracy than their results in lower-dimensional cases like Example~\ref{example:1} and Example~\ref{example:2}. However, we can still observe a better performance of the Gauss-Newton method than the other three training algorithms. The Gauss-Newton method demonstrates faster convergence and greater stability when compared to the Adam method, which exhibits numerous oscillations during training.

\begin{figure}[H]
	\centering
	\subfigure[Relative $L^2$-error of Newton-type algorithms]
    { \includegraphics[width=0.45\textwidth]{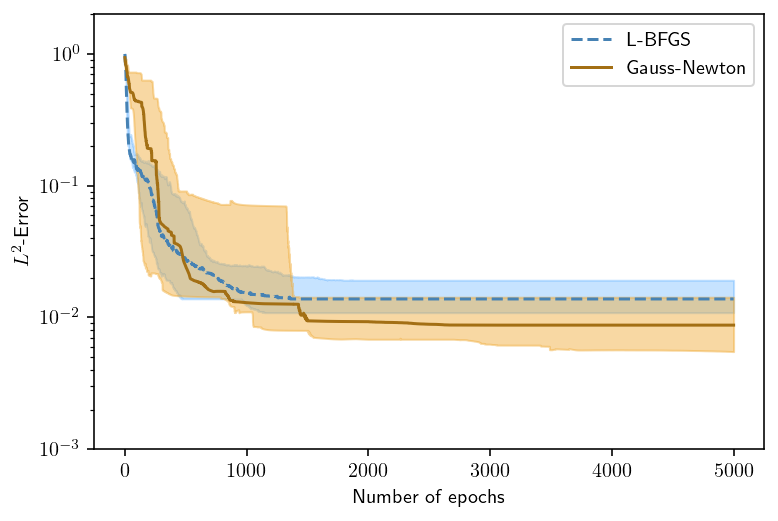}}
		\subfigure[Relative $L^2$-error of gradient-based algorithms]
	{ \includegraphics[width=0.45\textwidth]{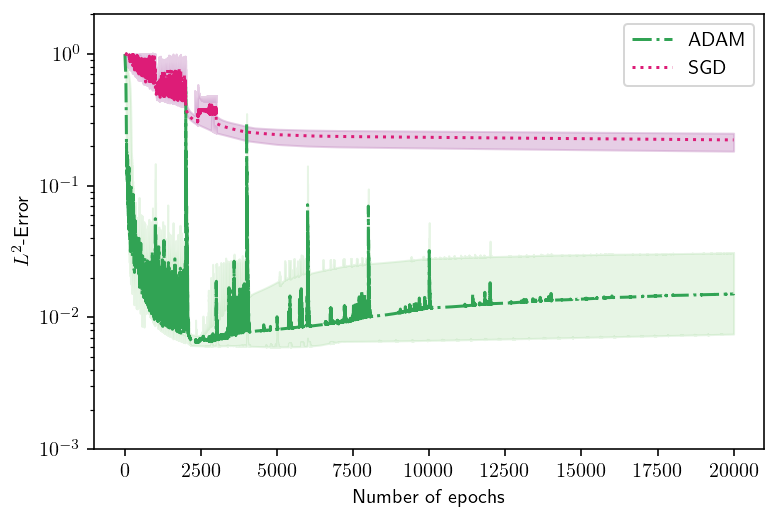}}\\
        \subfigure[Relative $H^1$-error of Newton-type algorithms]
	{ \includegraphics[width=0.45\textwidth]{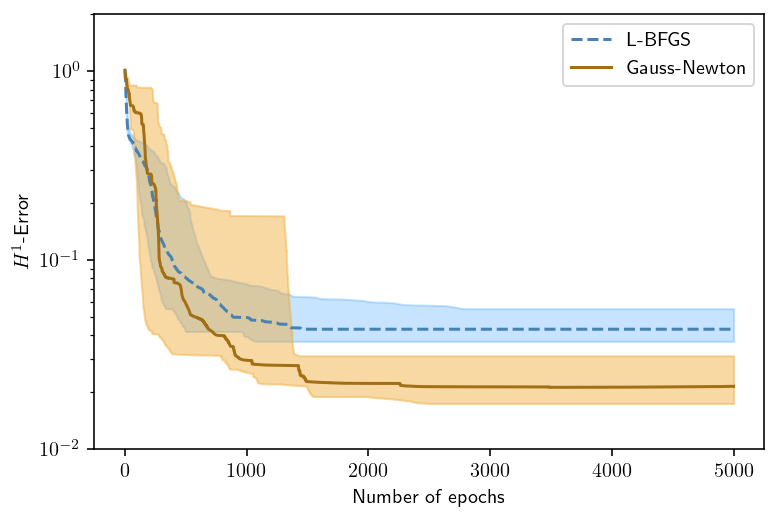}}
        \subfigure[Relative $H^1$-error of gradient-based algorithms]
	{ \includegraphics[width=0.45\textwidth]{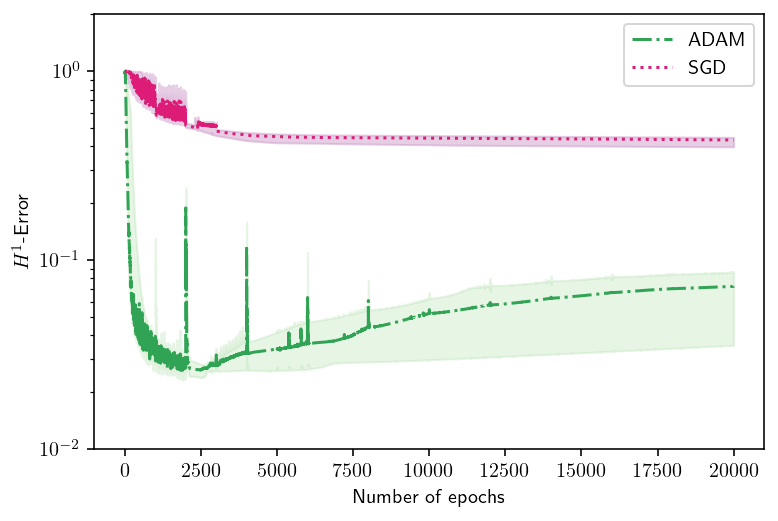}}
	\caption{Testing Relative Errors vs. Iterations. \textbf{Left:} Newton-type training algorithms. The Gauss-Newton method is employed with a back-tracking strategy to determine the learning rate, while L-BFGS is applied with the strong Wolfe condition. \textbf{Right:} Gradient-based algorithms. Both ADAM and SGD are applied with an initial learning rate of $1 \times 10^{-3}$, which is then halved every 2000 epochs until it reaches $1 \times 10^{-5}$.
 }\label{fig:loss_5d}
\end{figure}

}

\section{Conclusions}\label{con}
In this paper, the Gauss-Newton method has been introduced as a powerful approach for solving partial differential equations (PDEs) using neural network discretization in the variational energy formulation. This method offers significant advantages in terms of convergence and computational efficiency. The comprehensive analysis conducted in this paper has demonstrated the superlinear convergence properties of the Gauss-Newton method, positioning it as another choice for numerical solutions of PDEs. The method converges to semi-regular zeros of the vanishing gradient, indicating its effectiveness in reaching optimal solutions.

Furthermore, we provide the conditions under which the Gauss-Newton method is identical for both variational and L2 minimization problems. Additionally, a variant of the method known as the random Gauss-Newton method has been analyzed, highlighting its potential for large-scale systems.
The numerical examples presented in this study reinforce the efficiency and accuracy of the proposed Gauss-Newton method.

In summary, the Gauss-Newton method introduces a novel framework for solving PDEs through neural network discretization. Its convergence properties and computational efficiency make it a valuable tool in the field of computational mathematics. Future research directions may involve incorporating fast linear solvers into the Gauss-Newton method to further enhance its efficiency. This advancement would contribute to the broader utilization of the method and accelerate its application in various domains.

{\bf Data availability} Data sharing is not applicable to this article as no datasets were generated or analyzed during
the current study.

{\bf Declarations}
 WH and QH are supported by NIH via 1R35GM146894.  There is no conflict of interest.

{\bf Acknowledgment}
We thank Dr. Yiran Wang from Purdue University for valuable discussions and contributing efforts to the first version of numerical experiments.

\bibliographystyle{amsplain}
\bibliography{GaussNewton}

%\begin{align}
%\tau^{k+1} \leq \tau^k (h\delta + \epsilon)
%\end{align}
%\begin{align}
%\tau^k \leq (h\delta + \epsilon)^k\tau^0
%\end{align}
%\begin{align}
%\tau^{k+1} &\leq h((\tau^k)^2) + (\epsilon)^k\tau^0 \\
%&\leq h(h(\tau^{k-1})^2 + \epsilon^{k-1}\tau^0)^2 + (\epsilon)^k\tau^0 \\
%& \leq h^3 (\tau^{k-1})^4 + 2h^2 (\tau^{k-1})^2 \epsilon^{k-1}\tau^0 + h(\epsilon^{k-1}\tau^0)^2 + (\epsilon)^k\tau^0 \\
%&\leq \cdots \\
%&\lesssim h^{2^k-1} (\tau^0)^{2^k} + \tau^0 \left( \epsilon^k + \epsilon^{2(k-1)} + \epsilon^{3(k-1)} + \cdots\right) \\ 
%& \lesssim  h^{2^k-1} (\tau^0)^{2^k} + (2\epsilon)^k \tau^0
%\end{align}
%We might come across coefficients such as
%\begin{align}
%h^{2^j} C_{2^j}^{l} = h^{2^j} \dfrac{2^j !}{l! (2^j-l)!	} \leq h^{2^j} 2^j! \leq C 
%\end{align}
%where $C$ is independent of $l$ or $j$.

\newpage
\section*{Appendix}
In this section, we will introduce some preliminaries for the convergence analysis of Gauss-Newton's method. It is worth noting that the following results do not specify the form of activation functions.

To begin, let us consider the minimization problem (\ref{minEnergy}). We seek to find a minimizer $v$ of the functional $\mathcal{J}(w) = a(w,w) - (f,w)$ over a space $V$ of admissible functions. Here, $a(w,z)$ is a symmetric bilinear form defined as $\displaystyle a(w,z) = \int_{\Omega} \left(\dfrac{a}{2} \nabla w \cdot \nabla z + \dfrac{c}{2} wz\right) dx$.

\begin{lemma}\label{lemma1}
Suppose $v$ is the minimizer of (\ref{minEnergy}) or the solution of (\ref{2ndPDE}), and define $\Vert u \Vert_a^2 = a(u,u)$. Then, we have \begin{align}
\mathcal{J}(u) - \mathcal{J}(v) = \Vert u-v \Vert_a^2,~ \forall u\in V.
\end{align}
\end{lemma}
\begin{proof}
For any $u\in V$, we have the optimality condition for the minimizer $v$:
\begin{align}
0 = \delta \mathcal{J}(v)[u] = 2a(v,u) - (f,u), \quad \forall u \in V.
\end{align}

By a simple computation, we can show that
\begin{align*}
\mathcal{J}(u) - \mathcal{J}(v) &= a(u,u) - (f,u) - a(v,v) + (f,v)\\
& = a(u,u) - (f,u) + a(v,v) \\
& = a(u,u) - 2a(v,u) + a(v,v) \\
& = \Vert u-v \Vert_a^2.
\end{align*}
\end{proof}

%\begin{remark}
The energy minimization problem (\ref{minEnergy}) requires an admissible set with at least $H^1$ regularity, but different training algorithms may require more regularity. Therefore, the choice of the activation function in DNN models is crucial in ensuring that the admissible set has the required regularity. The $\text{ReLU}^k$ activation function with $k\geq 2$ is a suitable choice as it guarantees $H^k$ regularity, which is necessary for the gradient method to work. The $\text{ReLU}^k$-DNN with one hidden layer, denoted by $V_N^k = \lbrace \sum_{i=1}^{N} a_{i}\text{ReLU}^k\left(w_{i} x+b_{i}\right), a_{i}, b_{i} \in \mathbb{R}^{1}, w_{i} \in \mathbb{R}^{1 \times d} \rbrace$, then we have $V_N^k(\Omega) \subset H^k(\Omega), k\geq 1$.

By considering a general $J$ hidden layer $\text{ReLU}^k$-DNN, we denote the admissible set on $\Omega$ for the energy functional as $DNN_J$. To ensure that the DNN model we use is at least a $C^0$ function, we assume that $DNN_J \subset V = H^1$. We have an approximation result for $DNN_J$:

\begin{lemma}\label{lemma2}
For any given $\epsilon > 0$, there exist $J\in \mathbb{N}^{+}$ and $m\in \mathbb{N}^{+}$ such that $u(\cdot,\theta) \in DNN_{J}$, $\theta\in \mathbb{R}^m$, and $\theta^* = \arg\min_{\theta} \mathcal{J}(u(x,\theta))$ satisfies
\begin{align}
\Vert u(\cdot,\theta^*) - v \Vert_{L^2(\Omega)} & \lesssim \epsilon,\
\big | u(\cdot,\theta^*) - v \big |_{H^1(\Omega)}  \lesssim \epsilon.
\end{align}
\end{lemma}

We will use the notation ``$A \lesssim B$" to denote $A \leq C B$ for some constant $C$  independent of crucial parameters such as $\theta$.

\begin{proof}
By the Weierstrass theorem in Sobolev space $W^1_p$ {\bf \cite{de1959stone}} (for $1\leq p \leq \infty$) and the fact that $\text{ReLU}^k$-DNN functions are piecewise polynomials that belong to $H^1$, there exist $J\in \mathbb{N}^{+}$, $m\in \mathbb{N}^{+}$ and $\theta\in \mathbb{R}^m$ such that $u(\cdot,\theta)\in DNN_{J}$ and
\begin{align}
\Vert u(\cdot,\theta) - v \Vert_a \lesssim \Vert u(\cdot,\theta)  - v \Vert_{H^1(\Omega)} < \epsilon.
\end{align}
In particular, when considering $\text{ReLU}^1$-DNN functions, one can refer to \cite{he2020relu} for estimations of a sufficiently large depth $J$ and a sufficiently large number of neurons required to represent all the piecewise linear functions in $\mathbb{R}^d$. 

Next, as $DNN_{J} \subset V$, we have
\begin{align}
\mathcal{J}(u(\cdot,\theta)) \geq \mathcal{J}(u(\cdot,\theta^)) \geq \mathcal{J}(v),
\end{align}
which, together with Lemma~\ref{lemma1}, implies that
{
\begin{align*}
\Vert u(\cdot,\theta) - v \Vert_{L^2(\Omega)} \lesssim \Vert u(\cdot,\theta) -v \Vert_a \leq \Vert u(\cdot,\theta) -v \Vert_a &< \epsilon \\
\big | u(\cdot,\theta) - v \big |_{H^1(\Omega)} \lesssim \Vert u(\cdot,\theta^*) -v \Vert_a \leq \Vert u(\cdot,\theta) -v \Vert_a &< \epsilon.
\end{align*}
}
\end{proof} 

Based on the lemmas above, we can estimate the residual function as follows:
\begin{lemma}\label{lemma3}
For $\forall \epsilon > 0$, there exist $\delta > 0$, $J\in \mathbb{N}^{+}$ and $m\in \mathbb{N}^{+}$, 
such that $\theta\in \mathbb{R}^m$, $u(\cdot,\theta) \in DNN_J\subset H^1(\Omega)$, and $
\Vert \theta-\theta^* \Vert < \delta$, it holds 
\begin{align}\label{lemma3_1}
\Vert u(\cdot,\theta) -v \Vert_{L^2(\Omega)} <\epsilon 
\end{align}
and
\begin{align}\label{lemma3_2}
\int_{\Omega} \left(\nabla u(x,\theta) \cdot \nabla w(x) + u(x,\theta) w(x) - f(x) w(x) \right) dx < C \epsilon,\quad \forall w \in H^1(\Omega)
\end{align}
where $C >0$ is independent of $\theta$.
\end{lemma}
\begin{proof}
By the continuity of $u(\cdot,\theta)$ with respect to $\theta$, it is easy to see from Lemma~\ref{lemma2} that 
\begin{align}\label{lemma3one}
\Vert u(\cdot,\theta) -v \Vert_{L^2(\Omega)} \leq \Vert u(\cdot,\theta)-u(\cdot,\theta^*) \Vert_{L^2(\Omega)} + \Vert u(\cdot,\theta^*) - v \Vert_{L^2(\Omega)} < \epsilon.
\end{align}
Next, we have for $\forall w \in H^1(\Omega)$,
\begin{align}
\left( r(u(\cdot,\theta)), w \right)_{L^2(\Omega)} &= \int_{\Omega} \left(a\nabla u(x,\theta) \cdot \nabla w(x) + c u(x,\theta) w(x) - f(x) w(x) \right) dx\nonumber \\ 
&= \int_{\Omega} \left( a\nabla (u(x,\theta) - v(x)) \cdot \nabla w(x) + c (u(x,\theta) - v(x)) w(x) \right) dx\nonumber\\ 
& \lesssim \big | u(\cdot,\theta) -v \big |_{H^1(\Omega)} \big | w \big |_{H^1(\Omega)} + \Vert u(\cdot,\theta) -v \Vert_{L^2(\Omega)} \Vert w \Vert_{L^2(\Omega)}. \nonumber
\end{align}
Since $w \in H^1(\Omega)$, we can conclude form Lemma \ref{lemma2} that the estimation (\ref{lemma3_2}) is valid.
\end{proof}

Next, we have Wedin's Theorem about a perturbed matrix.

\begin{lemma}\label{lemma5} (Wedin, see \cite{wedin1973perturbation} or \cite{castro2016multiplicative})
Let $A\in \mathbb{R}^{m\times n}$ and $B = A+E$ with $rank(B) = rank(A)$. Then in any unitary invariant norm $\Vert \cdot \Vert$,
\begin{align}\label{lemma5one}
\Vert B^{\dagger} - A^{\dagger} \Vert \leq \mu \Vert A^{\dagger} \Vert_2 \Vert B^{\dagger} \Vert_2 \Vert E \Vert
\end{align}
for some moderate constant $\mu > 0$ that depends on the norm used. Moreover, if
\begin{align}\label{lemma5two}
\Vert E \Vert_2 < \dfrac{1}{\Vert A^{\dagger} \Vert_2},
\end{align} 
then the inverse of $B$ can be bounded by
\begin{align}\label{lemma5three}
\Vert B^{\dagger} \Vert_2 \leq \dfrac{\Vert A^{\dagger} \Vert_2}{1-\Vert A^{\dagger} \Vert_2\Vert E \Vert_2}
\end{align}
\end{lemma}

\begin{remark}
A unitarily invariant norm $\Vert \cdot \Vert$ satisfies $\Vert A \Vert = \Vert UAV \Vert$ for any unitary matrices $U$ and $V$. Specifically, the $2$-norm, the Frobenius norm, and many other norms related to the singular value are unitarily invariant norms.

\end{remark}

The following lemmas are just from \cite{zeng2021newton}.

\begin{lemma}\label{lemma2:zeng} For $\mathbb{F}=\mathbb{C}$ or $\mathbb{R}$, let $\mathbf{z} \mapsto \phi(\mathbf{z})$ be a continuous injective mapping from an open set $\Omega$ in $\mathbb{F}^n$ to $\mathbb{F}^m$. At any $\mathbf{z}_0 \in \Omega$, there is an open neighborhood $\Delta$ of $\phi\left(\mathbf{z}_0\right)$ in $\mathbb{F}^m$ such that, for every $\mathbf{b} \in \Delta$, there exists a $\mathbf{z}_{\mathbf{b}}$ in $\Omega$ and an open neighborhood $\Sigma_0$ of $\mathbf{z}_{\mathbf{b}}$ with
$$
\left\|\mathbf{b}-\phi\left(\mathbf{z}_{\mathbf{b}}\right)\right\|_2=\min _{\mathbf{z} \in \Sigma_0}\|\mathbf{b}-\phi(\mathbf{z})\|_2 .
$$
Further assume $\phi$ is differentiable in $\Omega$. Then
$$
\phi_{\mathbf{z}}\left(\mathbf{z}_{\mathbf{b}}\right)^{\dagger}\left(\mathbf{b}-\phi\left(\mathbf{z}_{\mathbf{b}}\right)\right)=\mathbf{0}.
$$
\end{lemma}
\begin{lemma}  (Stationary Point Property) Let $\theta \mapsto \nabla L(\theta)$ be a smooth mapping with a semiregular zero $\theta^*$ and $r={rank}\left(\mathbf{H}\left(\theta^*\right)\right)$. Then there is an open neighborhood $\Omega_*$ of $\theta^*$ such that, for any $\hat{\theta} \in \Omega_*$, the equality $\left(\mathbf{H}\right)(\hat{\theta})_{\text {rank-r }}^{\dagger} \nabla L (\hat{\theta})=\mathbf{0}$ holds if and only if $\hat{\theta}$ is a semiregular zero of $\nabla L$ in the same branch of $\theta^*$.
\end{lemma}

\begin{lemma} \label{lemma3:zeng} (Local Invariance of Semiregularity) Let $\theta^*$ be a semiregular zero of a smooth mapping $\nabla L$. Then there is an open neighborhood $\Delta_*$ of $\theta^*$ such that every $\hat{\theta} \in \Delta_* \cap (\nabla L)^{-1}(\mathbf{0})$ is a semiregular zero of $\nabla L$ in the same branch of $\theta^*$.
\end{lemma}

\end{document}